\documentclass{amsart}

\usepackage{nicefrac}
\usepackage{slashbox}
\usepackage{float}

\usepackage{amssymb}
\usepackage{graphicx} 

\usepackage{subfig}

\newtheorem{theorem}{Theorem}[section]
\newtheorem{lemma}[theorem]{Lemma}
\newtheorem{cor}[theorem]{Corollary}
\newtheorem{prop}[theorem]{Proposition}
\newtheorem{conj}{Conjecture}

\theoremstyle{definition}
\newtheorem{definition}[theorem]{Definition}

\theoremstyle{remark}

\numberwithin{equation}{section}

\newcommand{\Ref}[1]{(\ref{#1})}
\newcommand{\nn}{\nonumber \\}
\newcommand{\ds}{\displaystyle}

\newcommand{\abar}{a^{-1}}
\newcommand{\bbar}{b^{-1}}
\newcommand{\qbar}{q^{-1}}
\newcommand{\suba}[1]{\langle a^{#1} \rangle}

\newcommand{\Ex}{\mathbb{E}}
\def\W#1{\widetilde{#1}}

\newcommand{\LS}{MR0577064}

\newcommand{\AGLP}{MR2473819}

\newcommand{\Cohen}{Cohen}
\newcommand{\DykemaBS}{DykemaBS}
\newcommand{\Flajolet}{Flajolet}

\newcommand{\BartV}{MR2176547}
\newcommand{\Grig}{Grig}
\newcommand{\Kouksov}{MR1689726}
\newcommand{\ERW}{ERW}

\newcommand{\KouksovB}{MR1487319}

\newcommand{\BurClearyWiest}{MR2395786}
\newcommand{\GStald}{MR2415304}

\newcommand{\Mann}{MR2894945}
\newcommand{\Wagon}{MR1251963}

\newcommand{\MadrasSlade}{MR1197356}

\newcommand{\BergF}{MR634931}
\newcommand{\CarvCaracc}{MR702248}
\newcommand{\CarvEtAl}{MR690735}
\newcommand{\BuksMC}{janse2009monte}
\newcommand{\Metropolis}{metropolis}
\newcommand{\DeGennes}{DeGennes}
\newcommand{\Flory}{flory}
\newcommand{\TesiMonte}{tesimonte}
\newcommand{\Geyer}{geyer}
\newcommand{\Lips}{lipshitz}
\newcommand{\Stanley}{stanley1980}
\newcommand{\Kauers}{kauers}
\newcommand{\KauersGuess}{kauersguess}
\newcommand{\FlajoletContFrac}{flajoletcontfrac}
\newcommand{\Bellman}{bellman}
\newcommand{\Burillo}{burillo2001}
\newcommand{\Hammersley}{hammersley1982}

\newcommand{\BB}{MR2197808}
\newcommand{\TatchMoore}{TatchMoore}
\newcommand{\DykemaRW}{MR2216708}
\newcommand{\BartVirag}{MR2176547}
\newcommand{\BartKaimNekrash}{MR2730578}
\newcommand{\BartWoess}{MR2131635}
\newcommand{\Lalley}{MR2275700}
  
\newcommand{\NagA}{MR2087798}

\newcommand{\NagB}{MR1691645}
\newcommand{\DiacSCa}{MR1650316}
\newcommand{\DiacSCb}{MR1414925}
\newcommand{\DiacSCd}{MR1254308}
\newcommand{\DiacSCe}{MR1245303}
\newcommand{\DiacSCf}{MR1233621}

\newcommand{\OrtnerWoess}{MR2338235}
\newcommand{\Woess}{MR731608}

\begin{document}


\title{On trivial words in finitely presented groups}

\author[M. Elder]{M. Elder}
\address{School of Mathematical \& Physical Sciences, The~University~of~Newcastle, Callaghan, New South Wales, Australia}
\email{murray.elder@newcastle.edu.au}
\author[A. Rechnitzer]{A. Rechnitzer}
\address{Department of Mathematics, University of British Columbia, Vancouver, British Columbia, Canada}
\email{andrewr@math.ubc.ca}
\author[E.J.~Janse~van~Rensburg]{E.J.~Janse~van~Rensburg}
\address{York University, Toronto, Ontario, Canada}
\email{rensburg@mathstat.yorku.ca}
\author[T. Wong]{T. Wong}
\address{Department of Mathematics, University of British Columbia, Vancouver, British Columbia, Canada}
\email{twong@math.ubc.ca}



\date{\today}

\dedicatory{To the memory of Herb Wilf, Pierre Leroux and Philippe Flajolet ---\\
all great men of generating functions.}

\begin{abstract}
We propose a numerical method for studying the cogrowth of finitely presented
groups. To validate our numerical results we compare them against the
corresponding data from groups whose cogrowth series are known exactly. Further,
we add to the set of such groups by finding the cogrowth series for
Baumslag-Solitar groups $\mathrm{BS}(N,N) = \langle a,b | a^N b = b a^N \rangle$ and
prove that their cogrowth rates are algebraic numbers. 


\end{abstract}

\maketitle

\section{Introduction}
In this article we consider the function that counts the number of trivial words
in a finitely presented group, the so-called {\em cogrowth function}. The
exponential growth rate of this function is simply called the \emph{cogrowth}
and is intimately related to the {\em amenability} of the group  \cite{\Cohen,\Grig}. 
Amenability is an active area of current research, and cogrowth is just
one of many characterisations.  

In this article we propose a new numerical technique to estimate the cogrowth
of finitely presented groups, based on ideas from statistical mechanics, which
we show to be quite accurate in predicting the cogrowth  for a range of
groups for which the cogrowth series and/or  amenability is known: these include {\em
Baumslag-Solitar} groups, a finitely presented relative of the {\em basilica group}, and some free products studied by Kouksov \cite{\Kouksov}. 

The present article builds on previous work of a subset of the authors
\cite{\ERW}, where various techniques, also based in statistical mechanics, were
applied to the problem of estimating and computing the cogrowth  for
Thompson's group $F$. This in turn built on previous work of Burillo, Cleary and
Wiest \cite{\BurClearyWiest}, and  Arzhantseva, Guba,  Lustig, and Pr\'eaux \cite{\AGLP}, 
who applied experimental techniques to the problem
of deciding the amenability of $F$. 
In other work, Belk and Brown \cite{\BB} proved the currently best known upper bound for the isoperimetric constant for $F$, and 
Tatch Moore \cite{\TatchMoore}
 gives lower bounds on the growth rate of F\o lner function for $F$. 
 
 More generally a (by no means exhaustive) list of others working in the area of random walks on groups is
Bartholdi \cite{\BartKaimNekrash, \BartVirag, \BartWoess},  
              Diaconis and Saloff-Coste  \cite{\DiacSCd,\DiacSCb,\DiacSCa,\DiacSCe,\DiacSCf}, Dykema \cite{\DykemaBS,\DykemaRW},
 Lalley \cite{\Lalley},
Smirnova-Nagnibeda  \cite{\NagA, \NagB} and Woess \cite{\OrtnerWoess,\Woess}.


For the benefit of readers outside of group theory, we start with a precise
definition of group presentations and cogrowth.

\begin{definition}[Presentations and trivial words]
A presentation $$\langle a_1,\dots,a_k | R_1, \dots, R_m\rangle$$ encodes a group
as follows.
\begin{itemize}
\item The letters $a_i$ are elements of the group and are called {\em
generators}, and the $R_i$ are finite length words over the letters $a_1,\ldots, a_k,
a_1^{-1},\ldots, a_k^{-1}$ and are called {\em relations} or {\em relators}. 

\item  A group is called {\em finitely generated} if it can be encoded by a presentation with the list $a_1,\ldots, a_k$ finite, and {\em finitely presented} if it can be encoded by a  presentation with both  lists $a_1,\ldots, a_k$ and $R_1, \dots,
R_m$ finite.

 \item A word in the letters  $a_1,\ldots, a_k, a_1^{-1},\ldots, a_k^{-1}$  is
called {\em freely reduced} if it contains no subword $a_j^{\pm 1}a_j^{\mp 1}$. 

\item The set of all freely reduced words, together with the operation of
concatenation followed by free reduction (deleting $a_j^{\pm 1}a_j^{\mp 1}$
pairs) forms a group, called the {\em free group} on the letters
$\{a_1,\dots,a_k\}$, which we denote by  \\ $F(a_1,\ldots, a_k).$

\item Let $N(R_1, \dots, R_m)$ be the normal subgroup which contains all
words of the form $\ds \prod_{j=1}^m \rho_jR_j\rho_j^{-1}$ where $\rho_i$ is
any word in the free group, and $R_j$ is one of the relators or their inverses.
This subgroup is called the {\em normal closure} of the set of relators, and is
the smallest normal subgroup in $F(a_1,\ldots, a_k)$ that contains all words
$R_1, \dots, R_m$.

\item The group encoded by the presentation $\langle a_1,\dots,a_k | R_1, \dots,
R_m\rangle$  is defined to be the quotient group $F(a_1,\ldots, a_k)/N(R_1,
\dots, R_m)$. 

\item It follows that words in $F(a_1,\ldots, a_k)$ equals the identity element
in $G$  if and only if it lies in the normal subgroup $N(R_1, \dots, R_m)$, 
and so is equal to a product of conjugates of relators and their inverses.
\end{itemize}
\end{definition}
We will make extensive use of this last point in the work below. We call a word
in $F(a_1,\ldots, a_k)$ that equals the identity element in $G$ a {\em trivial
word}.

The function $c:\mathbb N\rightarrow \mathbb N$ where $c(n)$ is the number of
freely reduced words in the generators of a finitely generated group that
represent the identity element is called the {\em cogrowth function} and the
corresponding generating function is called the \emph{cogrowth series}. The
rate of exponential growth of the cogrowth function is the {\em cogrowth}
of the group (with respect to a chosen finite  generating set). Grigorchuk and
independently Cohen \cite{\Cohen, \Grig} proved that a finitely generated group
is  amenable if and only if its cogrowth is twice the number of generators
minus~1.

For more background on amenability and  cogrowth see \cite{\Mann, \Wagon}.
The free group on two (or more) letters, as defined above, is
known to be non-amenable. Also, subgroups of amenable groups are also amenable.
It follows that if a group contains a subgroup isomorphic to the free
group on 2 generators ($F(a_1,a_2)$ above), then it cannot be amenable.

The article is organised as follows. In Section~\ref{sec:numerical} we adapt an
algorithm designed to sample self-avoiding polygons  to the problem of
estimating the growth rate of trivial words in finitely presented 
groups. The algorithm we describe also works when the group is finitely
generated but has infinitely many relations -- in this case its application is
more subtle (in the way one samples relators from an infinite list). To 
validate the accuracy of our algorithm we test it on groups whose cogrowth
series are known exactly. In Section~\ref{sec:BSG} we add to this pool of
results by finding the cogrowth series of the Baumslag-Solitar groups $\mathrm{BS}(N,N) =
\langle a,b | a^N b = b a^N \rangle$.
We apply our algorithm and analyse the resulting data in
Section~\ref{sec:analysis} and summarise our results in
Section~\ref{sec:conclusion}.


\section{Metropolis Sampling of Freely Reduced Trivial Words in Groups}
\label{sec:numerical}

Our algorithm will be designed to sample words in a group $G$ along a Markov
chain using the Metropolis algorithm \cite{\Metropolis}. States will be sampled
by the algorithm by generating new states from a current state via
\textit{elementary moves}. These elementary moves will be defined in more detail
below -- they are local changes made in a systematic manner to a freely reduced
trivial word $w$ to obtain a new freely reduced trivial word $v$.

The approach is as follows: Let $w_n$ be the \textit{current state} of the
algorithm (so that $w_n$ is a freely reduced trivial word of $G$). Choose an
elementary move from a set of available elementary moves and create a trial word
$w_n^\prime$ by implementing the elementary move on $w_n$ (where $w_n^\prime$ is
also a freely reduced trivial word).  Accept $w_n^\prime$ as the next state in
the Markov chain with probability $P(\omega_n \to \omega_n^\prime)$, in which
case the next state is $w_{n+1} = w_n^\prime$.  If $w_n^\prime$ is rejected, then the next
state is by default $w_{n+1} = w_n$. This rejection technique is characteristic
of the Metropolis algorithm and it ensures that the sampling is aperiodic.

This implementation samples words $\{w_n\}$ for $n=0,1,2,\ldots$ along a Markov
chain which is initiated at a state $w_0$.  The initial state $w_0$ may be
chosen arbitrarily, but must be a freely reduced trivial word of $G$. It is
convenient to choose $w_0$ from the set of relators of $G$.

 The dynamical implementation of 
 our algorithm is inspired by the BFACF
algorithm \cite{\CarvCaracc, \CarvEtAl, \BergF} in statistical physics, which was developed to sample
lattice self-avoiding walks and polygons from stretched Boltzmann distributions.
The self-avoiding walk is a model of polymer entropy, a celebrated unsolved
problem in polymer physics and chemistry \cite{\DeGennes, \Flory}.  Details
about the implementation of the BFACF algorithm can also be found in 
\cite{\BuksMC, \MadrasSlade}.

\subsection{Elementary moves for sampling trivial words in a group $G$}

Let $G = \langle a_1,a_2,\dots a_k | R_1, R_2, \dots ,R_m ,\dots\rangle$ be a
group on $k$ generators with finite length relators $R_i$.  The number of
relators may be finite, or infinite. Let $w$ be a freely reduced trivial word in
$\{a_1,a_1^{-1},\dots,a_k, a_k^{-1} \}$. Denote the length of $w$ by $|w|$. And
finally, let $S$ be the set of relators $R_i$, their inverses $R_i^{-1}$ and all
cyclic permutations of relators and their inverses. Note that  $S$ is an
infinite set if and only if $G$ has 
an infinite set  of  relators.

Suppose that we have sampled along a Markov chain $\{ w_n \}$ and that
the current state is a freely reduced trivial word $w=w_n$ of length
$|w|=|w_n|$.  A new trivial word $w^\prime$ is constructed from $w$ by choosing
from the following two elementary moves:
\begin{itemize}
  \item \textit{Conjugation} --- Let $x$ to be one of the $2k$ possible 
generators (and their inverses) chosen uniformly and at random.  Put $w^\prime =
x w x^{-1}$ and perform free reductions on $w^\prime$ to
produce $w^{\prime\prime}$. 

\item \textit{Insertion} --- Let $R\in S$ be one of the relators or their
inverses or any cyclic permutations of those relators or their
inverses%
\footnote{For example, in $\mathrm{BS}(2,3)$ defined in Section~\ref{sec:BSG}, the
relator $a^2ba^{-3}\bbar$
yields $2 \times 7 = 14$ possible choices.}. %
 Choose an integer $m\in\{0,1,\ldots,|w|\}$ with uniform probability and
partition $w$ into two subwords $u$ and $v$, with $|u| = m$. Form $w^\prime = u
R v$, and freely reduce this word to get $w^{\prime\prime}$.  If $m=0$,
then $R$ is prepended to $w$, and if $m=|w|$, then $R$ is appended to $w$.
\end{itemize}
The elementary moves are implemented by choosing a conjugation with
probability $p_c$, and otherwise an insertion.

The two elementary moves produce freely reduced trivial words
$w^{\prime\prime}$ by acting on $w$.  A Metropolis style Monte Carlo algorithm
can be implemented using these moves provided that they are uniquely reversible.

One may verify that conjugations are uniquely reversible.  Unfortunately,
insertions are not, and this must be accounted for in the implementation of
the algorithm by conditioning the use of insertion moves such that they
become uniquely reversible.

We show by example that insertions are not reversible:  Let $R\in S$
and consider the insertion of $R^{-1}$ to the right of $R$ 
in the word $a^{\ell} R a^{-\ell}$.
This will reduce the word to the empty word, but there is no elementary
move which will produce $a^{\ell} R a^{-\ell}$ from the empty word by
inserting any relator on the empty word (here we assume $a^{\ell} R a^{-\ell}$ are not relators).
 This difficulty can be overcome
by rejecting proposed moves as a result of inserting $R$ if it changes
the length of the word by more than $|R|$.

A second difficulty may arise with insertions, and we show again by example
that an insertion may not be uniquely reversible, even if it it changes the
length of a word by at most $|R|$.  Consider the group $\mathbb{Z}^2=\langle a,b \ | \ ba\bbar\abar\rangle$
and insert the relator $R=ba\bbar\abar$ into the word 
$uba^{-1}b^{-1}aba^{-1}v$ at the position marked by $*$ below:
\begin{align}
u b \abar\bbar*ab \abar v & \mapsto u b \abar\bbar \cdot
ba\bbar\abar \cdot ab \abar v \mapsto ub\abar v
\end{align}
This move can be reversed in 2 ways. Insert $b \abar \bbar a$ (which
is a cyclic permutation of the inverse of the relation) at the~$*$
\begin{align}
   u*b\abar v \mapsto u \cdot b \abar \bbar a \cdot b\abar v
\end{align}
or  insert $\bbar ab
\abar$ (another cyclic permutation of the inverse of the relation) at the~$*$
\begin{align}
   ub\abar*v \mapsto ub\abar \cdot \bbar ab \abar \cdot v
\end{align}
This will disturb the detailed balance condition required for Metropolis
style algorithms with the result that the algorithm will sample from an
incorrect stationary distribution.

We show how to account for the above by modifying the insertion move as 
follows: Reject all attempted insertions of $R \in S$ in a word if either there
are cancellations to the right, or if it changes the length of the word by more
than $|R|$.  Attempted insertions which neither cancel to the right, nor
change the length of the word by more than $|R|$ will be called
\textit{valid}, and we call an insertion a  \textit{left-insertion} if
cancellations of generators only occurs to the left and if the insertion is
valid.
\begin{itemize}
\item \textit{Left-Insertion} --- Let $R \in S$ be one of the relators
or their inverses or any cyclic permutation of those relators and their
inverses. Choose an integer $m\in\{0,1,2,\ldots,|w|\}$ uniformly and 
partition $w$ into two subwords $u$ and $v$, with $|u| = m$. If $m=0$ then
prepend $w^{\prime} = R w$ and note that this is valid only if there are no
cancellations of generators.  If this is valid, then put $w^{\prime\prime}
= w^\prime$, otherwise put $w^{\prime\prime}=w$. If $m=|w|$, then append
$w^\prime = wR$ and this is valid even if there are cancellations
to the left.  Freely reduce $w^\prime$ to obtain $w^{\prime\prime}$.
Otherwise, form $w^\prime = u R v$. If $R$ cancels to the right with
$v$ then reject the proposed move and keep $w$. Otherwise, freely reduce
$w^\prime$ to obtain $w^{\prime\prime}$. If $|w^{\prime\prime}|< |w|-|R|$ then
reject the move (and keep $w$).
\end{itemize}
Left-insertions are uniquely reversible, and are suitable as an 
elementary move in a Metropolis style Monte Carlo algorithm for sampling freely
reduced trivial words in $G$.

\begin{lemma}
   Left-insertions are uniquely reversible.
\end{lemma}
\begin{proof}
Let $w = uv$ be a freely reduced trivial word in the group $G$ and let 
$R\in S$. Form $w^\prime = u R v$ via a left-insertion, where $u$ or $v$ may be 
the empty word.
\begin{itemize}
  \item Suppose there are no possible cancellations to the left or right ---
then $w^{\prime\prime} = u R v$, and the move can be uniquely reversed by 
inserting $R^{-1}$ (which must also be a relator in the group) to the
\emph{right} of $R$.  This gives $u R R^{-1} v \mapsto uv$. Further
cancellations cannot occur because $w=uv$ was freely reduced. Note that this is
unique because any other insertion must cancel $R$, and to do so would require 
cancellations to the right and so would not be a left-insertion.

\item Suppose there are some cancellations to the left when $R$ is inserted
in $w$.  In particular, in this case one has $w = u^\prime s v$ and $R = 
\bar{s} t$ for some freely reduced words $u^\prime$, $s$ and $t$ (where $t$ may
be the empty word).  Inserting $R$ to the right of $s$ and freely reducing the
word gives $w^{\prime\prime} = u^\prime t v$ (and $t$ may be empty). This move
is uniquely reversible by inserting $R^{-1} =\bar{t} s$ to the \emph{right} 
of $t$. This gives $u^\prime t \bar{t} s v \mapsto u\prime s v= w$. No further
cancellations are possible because the original word was freely reduced.
Again, by a similar argument, this move is unique --- all other possibilities
require cancellation to the right.
\end{itemize}
\end{proof}
The conjugation and left-insertion elementary moves can be implemented
in a Metropolis algorithm to sample freely reduced trivial words in $G$.

\subsection{Metropolis style implementation of the elementary moves}

Conjugations and left-insertions may be used to sample along a Markov
chain in the state space of freely reduced trivial words of a group $G$
on $k$ generators. 

The algorithm is implemented as follows:  Let $p_c \in [0,1]$, $\alpha \in
\mathbb{R}$ and $\beta \in \mathbb{R}^+$ be parameters of the algorithm and
assume that $\beta$ is small.  As above, let $S$ be the set of all
cyclic permutations of the relators and their inverses and recall that $S$
may be finite or infinite.

Define $P$, a probability distribution over $S$, so that $P(R)$ is the
probability of choosing $R\in S$ with $\sum_{R\in S} P(R) = 1$. Further, assume
that $P(R)>0$ for all $R\in S$ and also that $P(R) = P(R^{-1})$ (we shall
eventually require these two conditions). In the case that $S$ is finite we
choose $P$ to be the uniform distribution, although we are free to choose other
distributions.

Suppose that $w_n$ is the current state, a freely reduced trivial word
produced by the algorithm, and inductively construct the next state 
$w_{n+1}$ as follows:
\begin{itemize}
\item[$\circ$] With probability $p_c$ choose a conjugation move and otherwise
choose a left-insertion.
\item[$\circ$] If the move is a conjugation, then choose one of the $2k$
possible conjugations randomly with uniform probability:  Say that the
pair $(c,c^{-1})$ is chosen where $c$ is a generator or its inverse.
Put $u = c w_n c^{-1}$ and freely reduce $u$ to obtain $w^\prime$.
Construct $w_{n+1}$ from $w^\prime$ and $w_n$ as follows:
\begin{align}
 w_{n+1} &= \begin{cases}
  w^\prime, & \mbox{ with probability 
    $p=\min\left\{1, \frac{(|w'|+1)^{1+\alpha}}{(|w|+1)^{1+\alpha}}
	  \beta^{ |w'|-|w| } \right\}$}; \\
  w_n & \mbox{ otherwise}.
             \end{cases}
\label{eqn2} 
\end{align}
\item[$\circ$] If the move is a left-insertion, then choose an element
$R\in S$ with probability $P(R)$.  Choose a location
$m\in \{0,1,2,\ldots,|w_n|\}$ in the word $w_n$ uniformly. This is the
location where the left-insertion will be attempted.  Attempt a left
insert of $R$ at the location $m$.  Construct $w_{n+1}$ as follows:
\begin{align}
 w_{n+1} &= \begin{cases}
    w_n, & \mbox{if the left-insertion of $R$ is not valid}; \\
    w^\prime, & \mbox{ if $R$ is valid and with probability
    $p=\min\left\{ 1, \frac{(|w^\prime|+1)^{\alpha}}{(|w|+1)^{\alpha}}
    \beta^{|w^\prime|-|w|} \right\}$}; \\
    w_n, & \mbox{ otherwise} .
             \end{cases}
\label{eqn3} 
\end{align}
\end{itemize}

Let $w$ and $v$ be two words and suppose that $v$ was obtained from $w$
by a conjugation as implemented above.  Then the transition probability
$P_r (w\to v)$ is given by
\begin{align}
P_r (w\to v) &= \frac{1}{2k} \left( 
\frac{(|v|+1)^{1+\alpha}}{(|w|+1)^{1+\alpha}}
                     \beta^{ |v|-|w| } \right)
\end{align}
since a conjugation is chosen uniformly from $2k$ possibilities, and provided
that $p< 1$ in equation~\Ref{eqn2}. Otherwise, the transition probability of
the reverse transition is $P_r(v\to w) = \nicefrac{1}{2k}$.  This, in
particular, shows the condition of detailed balance for conjugation moves:
\begin{align}
P_r (w\to v) &= \left( \frac{(|v|+1)^{1+\alpha}}{(|w|+1)^{1+\alpha}}
                     \beta^{ |v|-|w| } \right) P_r (v\to w)
\end{align}
which simplifies to the symmetric presentation
\begin{align}
(|w|+1)^{1+\alpha} \beta^{|w|} P_r (w\to v) &=
(|v|+1)^{1+\alpha} \beta^{ |v|}  P_r (v\to w) .
\label{eqn6} 
\end{align}

In the alternative case that $w$ and $v$ are two words and $v$ was obtained
from $w$ by a left-insertion of $R\in S$ as implemented above, the
transition probability is given by
\begin{align}
P_r (w\to v) = \frac{P(R)}{|w|+1}
                    \left( \frac{(|v|+1)^{\alpha}}{
                     (|w|+1)^{\alpha}} \beta^{ |v|-|w| } \right)
\end{align}
where the element $R\in S$ is selected with probability $P(R)$,
the location for the left-insertion of $R$ is chosen with probability
$1/(|w|+1)$, and we have assumed (without loss of generality) that $p<1$ in
equation~\Ref{eqn3}.

Similarly, the transition probability of $v\to w$ via a left-insertion of
$R^{-1}\in S$ is
\begin{align}
P_r (v\to w) &= \frac{P(R^{-1})}{|v|+1} .
\end{align}
This gives a second condition on the probability distribution $P$ over $S$,
namely that $P(R) = P(R^{-1})$ for all elements $R\in S$.  In this event a
comparison of the last two equations, and simplification, gives
\begin{align}
(|w|+1)^{1+\alpha} \beta^{|w|} P_r (w\to v) &=
(|v|+1)^{1+\alpha} \beta^{ |v|}  P_r (v\to w)
\label{eqn9} 
\end{align}
as a condition of detailed balance for left-insertions.  This is the 
identical condition obtained for conjugation in equation~\Ref{eqn6}.  The
above is a proof of the following lemma.

\begin{lemma}
Let $\{ w_n \}$ be a Markov chain in the state space of freely
reduced words in $G$, and suppose the transition of state $w_n$ to
$w_{n+1}$ is due to a transition by a conjugation move with probability
$p_c$, and due to a left-insertion with probability $q_c=1-p_c$.  Then the
Markov chain samples from the stationary distribution
\begin{align*}
P_r (w) &= \frac{(|w|+1)^{1+\alpha} \beta^{|w|} }{\mathcal{N}}
\end{align*}
over its state space, where $\mathcal{N}$ is a normalising factor.
\label{lemma2} 
\end{lemma}

\begin{proof}
This lemma is a corollary of the Perron-Frobenius theorem (see
\cite{\Bellman} for example), and follows by summing the conditions of detailed
balance in equations~\Ref{eqn6} and~\Ref{eqn9} over $v$.
\end{proof}

\subsection{Irreducibility of the elementary moves}

In this subsection we examine the state space of the Markov chain in
Lemma~\ref{lemma2} by determining the irreducibility class of trivial freely
reduced words in $G$ with respect to the elementary moves of the algorithm.

The elementary moves above may be represented as a multigraph $M$ on the freely
reduced words of $G$:  Two freely reduced words $w,v$ form an arc $wv$ for each
elementary move (a conjugation or a left-insertion) which takes $w$ to $v$. 
Since each elementary move is uniquely reversible, $M$ may be considered
undirected. The irreducibility class of a freely reduced trivial word $w$ in $G$
is the collection of freely reduced trivial words in the largest connected
component $M_w$ of $M$ which contains $w$. The algorithm will be said to be
irreducible on freely reduced trivial words in $G$ if the words in $M_w$ form
exactly the family of freely reduced trivial
words in $G$.

\begin{lemma}
  Consider the group $G = \langle a_1,\dots,a_k | R_1 \dots R_m \ldots 
\rangle$ with $k$ generators. If $0<p_c<1$ and $P(R)>0$ for all $R \in S$, then 
the elementary moves defined above are irreducible on the set of all freely
reduced trivial words in that group.
\label{lemma3} 
\end{lemma}
\begin{proof}
Consider a relator of $G$, say $R_1\in S$.  Observe that left-insertions
can be used to change $R_1$ into any other relator $R_m$ of $G$.  Hence,
all the relators $R_m$ of $G$ are in the irreducibility class $M$ of $R_1$.
It follows that all cyclic permutations of the $R_m$, and inverses and
their cyclic permutations are also in $M$.  Hence, $S \subseteq M$.

Next, let $C = \langle w_n \rangle$ be a realisation of a Markov chain
with initial state $w_0 = R_1$.  All words $w_n$ sampled by C are obtained
by conjugation or by left-insertions by elements of $S$, and so they are
all trivial and freely reduced.  Thus $C\subseteq M$ if $C$ is initiated
by $R_1$.

It remains to show that any trivial and freely reduced word can occur
in a realisation of a Markov chain $C$ with initial state $R_1 \in S$.

A word $w\in\{a_1^{\pm 1}, \ldots, a_k^{\pm 1}\}^*$ represents the
identity element in the group if and only if it is the product of
conjugates of the relators $R_i^{\pm 1}$.  So $w$ is the word
\begin{align*}
 \displaystyle \prod_{j=1}^s\rho_jr_j\rho_j^{-1}
\end{align*}
after free reduction, where $\rho_j\in\{a_1^{\pm 1}, \ldots, a_k^{\pm 1}\}^*$
and $r_j=R_{i_j}^{\pm 1}$.

We can obtain $w$ using conjugation and left-insertion as follows:
\begin{itemize}
\item set $w=r_1$;
\item conjugate by $\rho_2^{-1}\rho_1$ one letter at a time to obtain
$w=\rho_2^{-1}\rho_1r_1\rho_1^{-1}\rho_2$ after free reduction;
\item insert (append) $r_2$ on the right;
\item repeat the previous two steps (conjugating by
$\rho_{j+1}^{-1}\rho_j$ then inserting $r_j$ on the left) until $r_s$
is inserted;
\item conjugate by $\rho_s$.
\end{itemize}
Since we only ever append $r_j$ to the end of the word, there are no
right cancellations, and at most $|r_j|$ left cancellations.

This completes the proof.
\end{proof}

\begin{cor}
The Monte Carlo algorithm is aperiodic, provided that $P(R)=P(R^{-1})>0$
and $0<p_c<1$.
\end{cor}

\begin{proof}
Let $P_r (w\to v)$ be the one step transition probability from state
$w$ to state $v$ in the Monte Carlo algorithm.  The probability of
achieving a transition $w\to v$ in $N$ steps is denoted by
$P_r^N (w\to v)$, and by Lemma~\ref{lemma3} there exists an $N_0$
such that $P_r^{N_0} (w\to v) > 0$, if both $w,v$ are freely
reduced trivial words.

The rejection technique used in the definition of both the conjugation
and left-insertion elementary moves immediately implies that if
$P_r^{N_0} (w\to v) > 0$ then $P_r^{M} (w\to v) > 0$ for all $M\geq N_0$.
Hence the algorithm is aperiodic.
\end{proof}

A Monte Carlo algorithm which is aperiodic, and irreducible on its state space, 
is said to ergodic. Hence, the algorithm above is ergodic on the state space of
freely reduced trivial words.  Under these conditions, the fundamental theorem 
of Monte Carlo Methods implies the algorithm samples along a Markov chain 
$C = \langle w_n \rangle$ asymptotically from the stationary distribution 
given in Lemma~\ref{lemma2}.

\subsection{Analysis of Variance}
The algorithm was implemented and tested for accuracy%
\footnote{As check on our coding, the algorithm was coded independently by two
of the authors (AR and EJJvR), and the results were compared. Further, we ran
the simulations making lists of observed trivial words for short lengths and
then compared these against exhaustive enumerations.}. %
The stationary distribution of the algorithm (see Lemma~\ref{lemma2}) shows
that the expectation value of the mean length of words sampled for given
parameters $(\alpha,\beta)$ is
\begin{align}
\Ex(|w|) &= \frac{\sum_w |w|\, (|w|+1)^{1+\alpha} \beta^{|w|}}{
\sum_w (|w|+1)^{1+\alpha} \beta^{|w|}}
\label{eqn10} 
\end{align}
where the summations are over all freely reduced trivial words in $G$.

We observe two points:  The first is that increasing $\beta$ will increase
$\Ex(|w|)$.  In fact, there is a critical point $\beta_c$ such that
$\Ex(|w|)<\infty$ if $\beta < \beta_c$, and $\Ex(|w|)$ is divergent if $\beta >
\beta_c$.  Observe that $\beta_c$ is independent of $\alpha$. The second point
is that increasing $\alpha$ will generally increase the value of $\Ex(|w|)$. 
This is convenient when one seeks to estimate the location of $\beta_c$.

Equation~\Ref{eqn10} is a log-derivative of the cogrowth series and will be
finite for $\beta$ below the reciprocal of the cogrowth (being the critical
point of the associated generating function) and divergent above it. Because of
this, we identify $\beta_c$ with the reciprocal of the cogrowth. Hence the
convergence of this statistic gives us a sensitive test of the cogrowth and so
the amenability of the group. For example, if the mean length of words sampled
from a group with 2 generators  at $\beta=\epsilon+\nicefrac{1}{3}$ is
finite, then the group is not amenable.

The realisation of a Markov chain $C = \{w_0,w_1,\ldots,w_n,\ldots\}$ by the
algorithm produces a correlated sequence of an observable (for example the
length of words).  We denote the sequence of observables by
 $\{O(w_1),O(w_2),O(w_3),\ldots,O(w_n)\}$ .  
The sample average of the observable over the realised chain is given by
\begin{align}
{\langle\hspace{-1mm}\langle} O(w) {\rangle\hspace{-1mm}\rangle}_n
&= \frac{1}{n} \sum_{i=1}^n O(w_i).
\label{eqn11} 
\end{align}
This average is asymptotically an unbiased estimator distributed normally
about the expected value~$\Ex(O(w))$, given by
\begin{align}
\Ex(O(w)) = \frac{\sum_w O(w) \, (|w|+1)^{1+\alpha} \beta^{|w|}}{
\sum_w (|w|+1)^{1+\alpha} \beta^{|w|}}.
\label{eqn12} 
\end{align}
Hence, ${\langle\hspace{-1mm}\langle} O(w) {\rangle\hspace{-1mm}\rangle}_n $ may
be computed to estimate the expected value $\Ex(O(w))$.

It is harder to determine the variance in the distribution of
${\langle\hspace{-1mm}\langle} O(w) {\rangle\hspace{-1mm}\rangle}_n $
about $\Ex(O(w))$.  Although the Markov chain produces a time series of
identically distributed states, they are not independent, and autocorrelations
must be computed along the time series to determine confidence intervals about
averages.

The dependence of an observable along a time series is statistically measured by
an autocorrelation function.  The autocorrelation function usually decays at an
exponential rate measured by the autocorrelation time $\tau_O$ along the time
series. In particular, the measured connected autocorrelation function of
the algorithm is defined by 
\begin{align}
S_O (k) &= {\langle\hspace{-1mm}\langle} O(w_i)
O(w_{i+k}){\rangle\hspace{-1mm}\rangle}_n
- {\langle\hspace{-1mm}\langle} O(w) {\rangle\hspace{-1mm}\rangle}_n^2 ,
\end{align}
and is dependent on $n$, the length of the chain.  If $n$ becomes very large,
then $S_O(k)$ measures the correlations between states a distance of
$k$ steps apart. The Markov chain is asymptotically homogeneous (independent
of its starting point); this implies that
${\langle\hspace{-1mm}\langle} O(w_i) {\rangle\hspace{-1mm}\rangle}_n
\simeq {\langle\hspace{-1mm}\langle} O(w) {\rangle\hspace{-1mm}\rangle}_n $
if both
$n$ and $i$ are large, and if $i\ll n$ . Thus, for large values of $n$ and $i$,
the autocorrelation time $S_O(k)$ is only dependent on the separation k between
the observables $O(w_i)$ and $O(w_{i+k})$.

Normally, the autocorrelation function of a homogeneous chain is expected to
decay (to leading order) at an exponential rate given by
\begin{equation}
S_O(k) \simeq C_O \, e^{-k/\tau_O}
\end{equation}
where $\tau_O$ is the exponential autocorrelation time of the observable $O$.
The autocorrelation time $\tau_O$ sets a time scale for the decay of
correlations in the time series $\{ O(w_i)\}$:  If $k>>\tau_O$,
then the states $O(w_i)$ and $O(w_{i+k})$ are for all practical statistical
purposes independent. These observations allow us to compute statistical
confidence intervals on the average 
${\langle\hspace{-1mm}\langle} O(w) {\rangle\hspace{-1mm}\rangle}_n$ 
in a systematic way.

Suppose that a time series of length $N$ of observables $\{ O(w_i) \}$ were
realised by the Markov chain Monte Carlo algorithm.   Cut the times series in
blocks of size $M\ll N$, but with $M\gg \tau_O$.  Then 
one may determine $\lfloor N/M \rfloor$ averages estimating
${\langle\hspace{-1mm}\langle} O(w) {\rangle\hspace{-1mm}\rangle}_n$
over the blocked data, given by 
\begin{align}
  [O(w)]_i &= \frac{1}{M} \sum_{j=1}^M O(w_{iM+j})
\end{align}
for $i=0,1,\ldots,\lfloor N/M \rfloor-1$.

The sequence of estimates 
$\{[O(w)]_0,[O(w)]_1,\ldots,[O(w)]_{\lfloor N/M\rfloor-1} \}$ is
itself a time series, and if these are independent estimates, then
for large $M \ll N$ its variance is estimated by determining
\begin{align}
s^2_{M,O} &= \langle [O]^2 \rangle - \langle[O]\rangle^2,
\end{align}
where canonical averages $\langle \cdot\rangle$ are taken.  So if the $[O(w)]_i$
are treated as independent measurements of $\Ex (O(w))$,  then the $67$\%
statistical confidence interval $\sigma_{M,O}$ is given by
\begin{align}
\sigma_{M,O}^2 &= \frac{s_{M,O}^2}{\lfloor N/M \rfloor - 1}.
\end{align}
In practical applications the above is implemented by increasing $M\ll 
N$ until $\sigma_{M,O}$ is insensitive to further increases. In this event one
has $M\gg \tau_O$, and $\sigma_{M,O}$ is the estimated 
$67$\% statistical confidence interval on the average computed in 
equation~\Ref{eqn11}.

In this paper we consider the average length of words -- that is,
$O(w) = |w|$ for each freely reduced and trivial word $w$ sampled by the
algorithm.  We use our algorithm to determine the \textit{canonical expected
length} of freely reduced trivial words \textit{with respect to the Boltzmann
distribution}.  This is defined by putting $\alpha=-1$ in equation~\Ref{eqn10}:
\begin{align}
\Ex_C(|w|) &= \frac{\sum_w |w|\, \beta^{|w|}}{\sum_w  \beta^{|w|}}
\label{eqn11A} 
\end{align}
where the summation is over all freely reduced trivial words in $G$, except the
empty word.

An estimator of $\Ex_C (|w|)$ is obtained by putting
$O(w)=|w|/(|w|+1)^{1+\alpha}$ and $O(w)=1/(|w|+1)^{1+\alpha}$ in
equation~\Ref{eqn12}.  This gives
\begin{align}
\Ex_C (|w|) &= \frac{\Ex\left(\frac{|w|}{(|w|+1)^{1+\alpha}}\right)}{
\Ex\left(\frac{1}{(|w|+1)^{1+\alpha}}\right)} .
\end{align}
In other words, for arbitrary choice of $\alpha$, the ratio estimator
\begin{align}
\langle |w| \rangle_n = \frac{
{\langle\hspace{-1mm}\langle}
|w|/(|w|+1)^{1+\alpha}{\rangle\hspace{-1mm}\rangle}_n
}{{\langle\hspace{-1mm}\langle}1/(|w|+1)^{1+\alpha}{\rangle\hspace{-1mm}\rangle}
_n}
\label{eqn rat est}
\end{align}
may be used to estimate the canonical expected length $\Ex_C (|w|)$ over
the Boltzmann distribution on the state space of freely reduced trivial
words in $G$.  This is particularly convenient, as one may choose
the parameter $\alpha$ to bias the sampling in order to obtain better
numerical results.  For example, it is frequently the case that
(long) trivial words in the tail of the Boltzmann distribution are sampled
with low frequency, and by increasing $\alpha$ the frequency may be
increased.  This gives larger sample sizes on long words, improving the
accuracy of the numerical estimates of the canonical expected
length of words.  For more details, see for example Section~14
in~\cite{\BuksMC}.

\subsection{Implementation}
The algorithm was implemented using a Multiple Markov chain Monte Carlo
algorithm \cite{\Geyer,\TesiMonte} --- an approach that is also known as as
parallel tempering. This greatly reduces autocorrelations in the realised Markov
chains and was achieved as follows:  Define a sequence of values of
$\beta$ such that $0 < \beta_1 < \beta_2 < \ldots < \beta_m < \beta_c$.
Separate chains are initiated at each of the $\beta_i$ and run in
parallel.  States at adjacent values of the $\beta_i$ are compared and
swapped.  This coupling of adjacent chains creates a composite Markov
chain, which is itself ergodic (since each individual chain is) with
stationary distribution the product distribution over all the
separate chains.  This implementation greatly increases the mobility
of the Markov chains, and reduces autocorrelations.  The analysis of
variance follows the outline above.  For more detail on a Multiple
Markov chain implementation of Metropolis-style Monte Carlo algorithms,
see \cite{\BuksMC} for example.

In practice we typically initiated 100 chains clustered towards
larger values of $\beta$ where the mobility of the algorithm is low. Each
chain was run for about $1000$ blocks, each block a total of $2.5\times 10^7$
iterations. The total number of iterations over all the chains were
$2\times 10^9$ iterations, which typically took about 1 week of CPU time
on a fast desktop linux station for each group we considered. We also ran each
group at five different $\alpha$ values $-1,0,1,2,3$. The larger values of
$\alpha$ will ensure that we sample into the tail of the distribution over
trivial words --- in practice the different $\alpha$ values gave consistent
results. Data were collected and analysed to estimate the cogrowth of each
group.

In the next sections, we compare our numerical results with exact
analysis of the Baumslag-Solitar groups.  This will demonstrate the
validity of our numerical approaches above.

\section{Exact cogrowth series for Baumslag-Solitar groups}
\label{sec:BSG}
\subsection{Equations}
Consider the Baumslag-Solitar group 
\[ \mathrm{BS}(N,M) = \langle a,b | a^N b = b a^M \rangle = \langle a,b
| a^N b a^{-M} \bbar \rangle .\] 
Our aim is to compute its cogrowth function, or the corresponding
generating function. Rather than obtain this directly, we instead consider
the set of words (they are not required to be freely reduced) which generate
elements in the horocyclic subgroup $\langle a \rangle$ --- let $\mathcal{H}$
be the set of such words. In what follows we will abuse notation and
when a \emph{word} $w$ generates an element in a subgroup $\langle a^k
\rangle$, we shall write $w \in \langle a^k \rangle$.

Any word in $\{a^{\pm1}, b^{\pm1}\}$ can be transformed into a normal-form for
the corresponding group element by ``pushing'' each $a$ and $\abar$ in the
word as far to the right as possible using the identities
$a^{\pm N} b = b a^{\pm M}$ and $a^{\pm M} \bbar = \bbar a^{\pm N}$, 
and replacing $a^{-i}b$ by $a^{N-i}a^{-N}b=a^{N-i}ba^{-M}$ 
(and similar for $a^{-j}b^{-1}$, where $0<i<N$ and $0<j<M$) so that only positive powers of $a$ appear before a $b^{\pm 1}$ letter.
 The resulting normal form can be written as $P a^k$, where $k$ is the
$a$-exponent, and $P$ is a word in the ``alphabet'' $\{b, ab, \dots a^{N-1}b,
\bbar, a\bbar, \dots a^{M-1}\bbar \}$ that we call the prefix  (see
\cite{\LS}  p. 181).

Consider a normal form $Pa^k$. 
\begin{itemize}
 \item Multiplying this on the right by $a^{\pm 1}$ results in $Pa^{k\pm 1}$. 
 \item If $k = N\ell$ then multiplying on the right by $b$ results in $Pb a^{M
\ell}$ --- if $P$ ends in a $\bbar$ then it will shorten and the $a$-exponent
will be updated accordingly.
 \item If $k = M\ell$ then multiplying on the right by $\bbar$ results in
$P\bbar a^{N \ell}$ --- if $P$ ends in a $b$ then it will shorten and the
$a$-exponent will be updated accordingly.
 \item Otherwise multiplying by $b^{\pm 1}$ will change the $a$-exponent and
lengthen the prefix.
\end{itemize}

Now define $g_{n,k}$ to be the number of words in $\mathcal{H}$ of length $n$
that generate the element with normal form $a^k$. Clearly we have $g_{n,k} =
g_{n,-k}$. Define the generating function
\begin{align}
  G(z;q) &= \sum_{n,k} g_{n,k} z^n q^k.
\end{align}
It is very convenient to define the following subsets of $\mathcal{H}$ and
their corresponding generating functions.
\begin{itemize}
 \item Let $\mathcal{L}$ be the set of words in $\mathcal{H}$ that cannot be
written as $uv$ where $u$ generates an element with normal from $\bbar
a^j$ for any $j$.
 \item Let $\mathcal{K}$ be the set of words in $\mathcal{H}$ that cannot be
written as $uv$ where $u$ generates an element with normal from $ba^j$ for any
$j$.
\end{itemize}
Let the generating functions of these words be $L(z;q)$ and $K(z;q)$
respectively. We note that $L(z;1) = K(z;1)$, since the inverse of any word in
$\mathcal{L}$ gives a word in $\mathcal{K}$ and vice versa. We then need to
define the operator $\Phi_{d,e}$ which acts on the above generating functions to
annihilate all powers of $q$ except those that have $a$-exponent 
equal to $0 \mod d$ and which maps them to powers of $0 \mod e$.
\begin{align}
 \Phi_{d,e} \circ \sum_n z^n \sum_{k} c_{n,k} q^k 
  &= \sum_n z^n \sum_j c_{n,dj} q^{ej}
\end{align}

With these definitions we can write down a set of equations satisfied by the
generating functions $G(z;q), K(z;q)$ and $L(z;q)$.
\begin{prop}
\label{prop GLK eqn}
The generating functions $G,K,L$ satisfy the following system of equations.
\begin{align*}
 L &=  1 + z(q+\bar{q})L + z^2 L\cdot
  \left[  \Phi_{N,M}\circ L + \Phi_{M,N} \circ K \right] 
   - z^2 \left[ \Phi_{M,N} \circ K \right] \cdot
  \left[ \Phi_{N,N} \circ L \right], \\
  K & = 1 + z(q+\bar{q})K + z^2 K \cdot
  \left[ \Phi_{M,N} \circ K + \Phi_{N,M} \circ L \right]
  - z^2 \left[ \Phi_{N,M} \circ L \right] \cdot 
  \left[ \Phi_{M,M} \circ K \right],
\intertext{and}
  G &= 1 + z(q+\bar{q})G 
  + z^2 G \cdot \left[ \Phi_{N,M} \circ L +  \Phi_{M,N} \circ K \right] 
\end{align*}
where we have written $G \equiv G(z;q)$ etc.
\end{prop}
We remark that these equations can be transformed into equations for the
cogrowth series by substituting $z \mapsto \frac{t}{1+3t^2}$ and replacing each
generating function $f(z) \mapsto h(t)\left[ \frac{1+3t^2}{1-t^2}\right]$. 
We found it easier to work with the equations as stated.
\begin{proof}

\begin{figure}[h]
 \centering
\includegraphics[height=3cm]{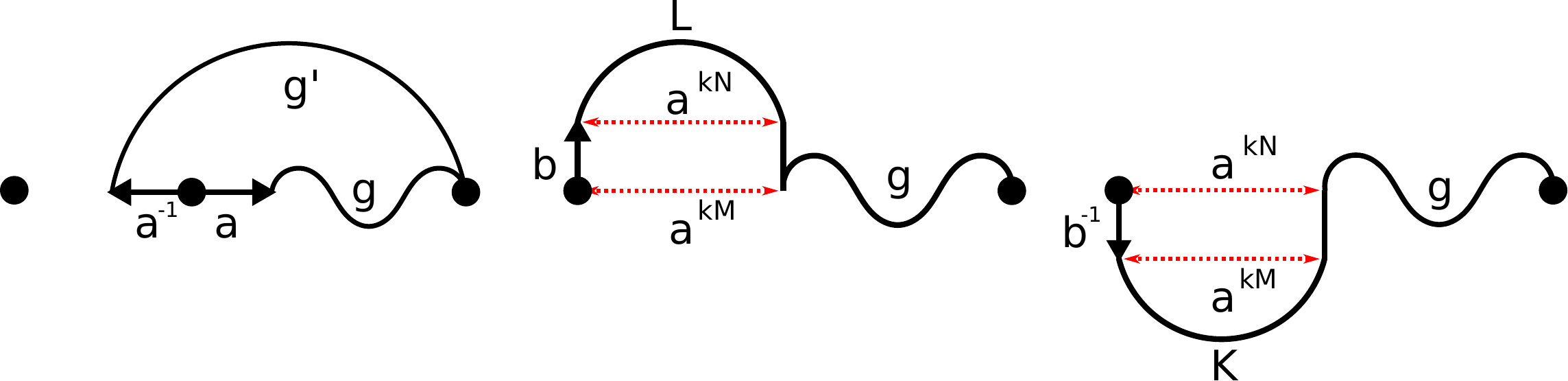}
\caption{Any word in $\mathcal{H}$ can be decomposed by considering its first
letter. There are 5 possibilities, falling into the 4 cases we have drawn
here. The subwords $g, g' \in \mathcal{H}$, $L\in\mathcal{L}$ and
$K\in\mathcal{K}$.}
\label{fig:decompG}
\end{figure}

First, we note that the set $\mathcal{H}$ is closed by prepending and appending
the generator $a$ and $\abar$. We factor $\mathcal{H}$ recursively by
considering the first letter in any word $w \in \mathcal{H}$ (see
Figure~\ref{fig:decompG}). This gives four cases:
\begin{itemize}
\item $w$ is the empty word.
\item The first letter is $a$ or $\abar{}$. Then $w = a v$ or $w = \abar
v$ for some $v \in \mathcal{H}$, increasing the length by $1$ and altering the
$a$-exponent by $\pm1$. At the level of generating functions this gives
$z(q+\qbar) G(z;q)$.
\item The first letter is $b$. Factor $w = u v$ where $u$ is the shortest
word so that $u \in \suba{}$. Thus, $u = b u' \bbar$ for some $u' \in
\suba{N}$. The minimality of $u$ ensures $u' \in \mathcal{L}$. Combined, this
gives $u \in \suba{M}$. At the level of generating functions, the maps words
counted by $z^nq^{kN}$ to $z^{n+2}q^{kM}$ and resulting in $z^2 \cdot \Phi_{N,M}
\circ L(z;q)$.
\item The first letter is $\bbar$. Factor $w = u v$ where $u$ is the
shortest word so that $u \in \suba{}$. As per the previous case, $u =
\bbar u' b$ for some $u' \in \suba{M}$ with $u' \in \mathcal{K}$. Combined, this
gives $u \in \suba{N}$. Similar reasoning gives $z^2 \cdot \Phi_{M,N} \circ
K(z;q)$
\end{itemize}

Now consider an element $w \in \mathcal{L}$, and we note that $\mathcal{L}$
(and $\mathcal{K}$) is closed under appending the generators $a$ and
$\abar$, but not prepending. See Figure~\ref{fig:decompL}. In a similar manner
to the above, we factor words in $\mathcal{L}$ recursively by considering
the last letter of $w$.
\begin{itemize}
\item $w$ is the empty word.
\item The last letter is $a$ or $\abar$. Then $w = v a$ or $w = v \abar$
for some $v \in \mathcal{L}$, increasing the length by $1$ and altering the
$a$-exponent by $\pm1$. This yields the term $z(q+\bar{q}) L(z;q)$.
\item The last letter is $\bbar$. Factor $w = u v$ where $u$ is the
longest subword such that $u \in \suba{}$ and $v$ is non-empty. This forces $v =
b v' \bbar$ with the restriction that $v' \in \mathcal{L}$. Since both $v,v' \in
\mathcal{L}$ we must have $v' \in \suba{N}$ and $v \in \suba{M}$, and this
yields $z^2 L(z;q) \cdot \Phi_{N,M} \circ L(z;q)$.
\item The last letter is $b$. Factor $w = u v$ where $u$ is the longest
subword such that $u \in \suba{}$ and $v$ is non-empty. This forces $v = \bbar{}
v' b$ with the restriction that $v' \in \mathcal{K}$. Further, $w \in
\mathcal{L}$ implies the subword $u\not\in \suba{N}$. Otherwise, $w \notin
\mathcal{L}$ as the subword $u\bbar$ generates an element with normal form
$\bbar a^j$ for some $j$. 

The generating function for $\left\{w \in \mathcal{L} \mid w \not\in
\suba{N}\right\}$ is given by $(L - \Phi_{N,N}\circ L)$, and so this last case
gives  $z^2 (L(z;q) - \Phi_{N,N} \circ L(z;q) ) \cdot \Phi_{M,N} \circ K(z;q)$.
\end{itemize}
Putting all of these cases together and rearranging gives the result.
The equation for $\mathcal{K}$ follows a similar argument.
\end{proof}
\begin{figure}[h]
 \centering
\includegraphics[height=3cm]{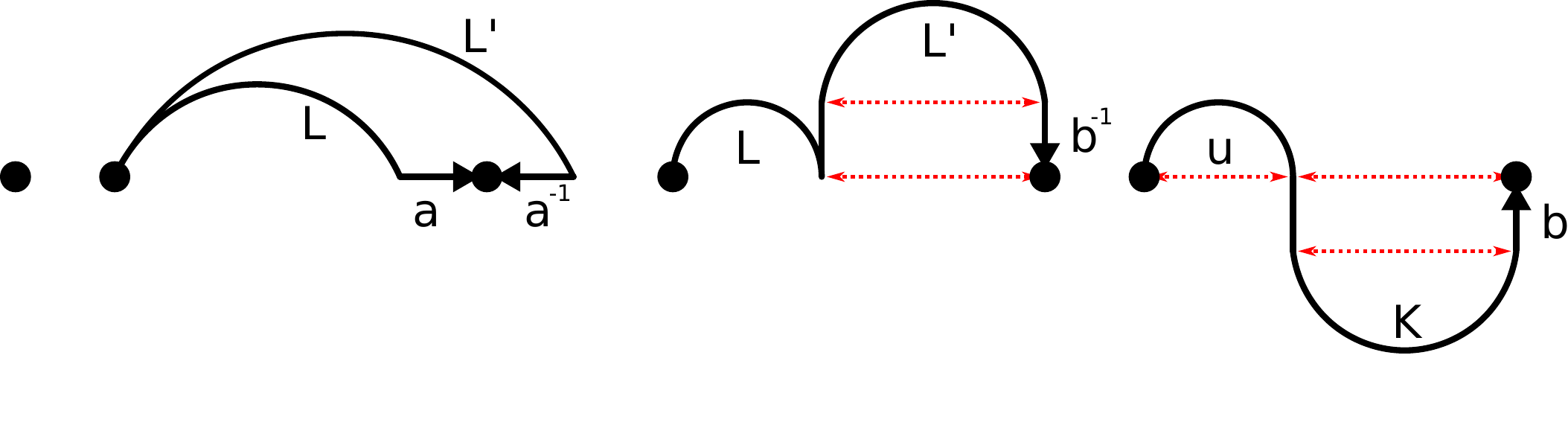}
\caption{Any word in $\mathcal{L}$ can be decomposed by considering its last
letter. This results in the four possible factorisations we have drawn
here. The subwords $L, L' \in\mathcal{L}$, $K\in\mathcal{K}$ and $u$ is a word
in $\mathcal{L}$ that generates an element in the subgroup $\langle a \rangle$,
but not in the subgroup $\langle a^N \rangle$.}
\label{fig:decompL}
\end{figure}

\subsection{Solution for $\mathrm{BS}(N,N)$}
\label{sec bsnn}
The number of trivial words of length $n$ in $\mathbb{Z}^2 \cong \mathrm{BS}(1,1)$ has
long been known to be $\binom{n}{n/2}^2$ (for even $n$)%
\footnote{Perhaps the easiest proof known to the authors is the following. Map
any trivial word to a path on the square grid. Now rotate the grid $45^\circ$
and rescale (by $\sqrt 2$). Each step now changes the $x$-ordinate by $\pm1$ and
similarly each $y$-ordinate by $\pm1$. In a path of $n$-steps, $n/2$ steps must
increase the $x$-ordinate and $n/2$ must decrease it and so giving
$\binom{n}{n/2}$ possibilities. The same occurs independently for the
$y$-ordinates and so we get $\binom{n}{n/2}^2$ possible trivial words.
}. %
This number grows as $4^{n+\nicefrac{1}{2}}/\pi n$,
and the factor of $n^{-1}$ implies that the corresponding generating function is
not algebraic (see, for example, section VII.7 of \cite{\Flajolet}). 
The generating function does satisfy a linear differential equation with
polynomial coefficients and so is D-finite \cite{\Stanley} (in fact it can be
written in closed form in terms of elliptic integrals).  The class of D-finite
functions includes rational and algebraic functions and many of the most famous
functions in mathematics and physics. Indeed, most of the known group growth and
cogrowth series are D-finite (being algebraic or rational\footnote{Kouksov proved that the cogrowth series is a rational function if and only if the group is finite \cite{\KouksovB}.}). We prove (below)
that when $N=M$, the cogrowth series is D-finite and we strongly suspect that
when $N \neq M$, the cogrowth series lies outside this class.
\begin{prop}
\label{prop NN GL eqn}
When $N = M$ the generating functions $K(z;q) = L(z;q)$ and the generating
functions $K=L, G$ satisfy 
 \begin{align*}
 L &=  1 + z(q+\bar{q})L + 2 z^2 L\cdot
  \left[  \Phi_{N,N}\circ L \right] 
   - z^2 \left[ \Phi_{N,N} \circ L \right]^2 \\
  G &= 1 + z(q+\bar{q})G 
  + 2 z^2 G \cdot \left[ \Phi_{N,N} \circ L \right]
\end{align*}
Further, these equations reduce to a set of algebraic equations in $G,L$ and
$\left[ \Phi_{N,N} \circ L \right]$. In particular if we write $L_0(z;q) =
\left[ \Phi_{N,N} \circ L \right]$, and let $\omega = e^{2\pi i / N}$ then we
have
\begin{align*}
  N L_0(z;q) &= \sum_{j=0}^{n-1} L(z; \omega^j q) =
  \sum_{j=0}^n \frac{ 1 - z^2 L_0(z;q)}{
    1 - z(\omega q + 1/\omega q) - 2z^2 L_0(z;q)
  }.
\end{align*}
\end{prop}
For example, for $\mathrm{BS}(2,2)$ the generating function $G(z;q)$ satisfies the
following cubic equation
\begin{align}
\label{eqn bs22 gzq}
1+3zQG - (1-4z^2-z^2Q^2)G^2 - zQ(1-zQ-2z)(1-zQ+2z)G^3 & = 0,
\end{align}
where we have written $Q = q + \bar{q}$.
\begin{proof}
The proof is a corollary of Proposition~\ref{prop GLK eqn}. Setting
$N=M$ simplifies the equations considerably and forces $K(z;q) = L(z;q)$.
We note that $L_0(z;q) = L_0(z;\omega q)$ and the equation for $L_0(z;q)$
follows. Hence both $L(z;q)$ and $G(z;q)$ are also algebraic.
\end{proof}

We are not interested in the full generating function $G$, rather we
are mainly interested in the coefficient of $q^0$.
\begin{cor}
For $\mathrm{BS}(N,N)$ the generating function $[q^0] G(z;q) = \sum g_{n,0}z^n$ is
D-finite. That is, it satisfies a linear generating function with polynomial
coefficients. Furthermore, the cogrowth series (being the generating function of
freely reduced words equivalent to the identity) is also D-finite.

It follows that the cogrowth of $\mathrm{BS}(N,N)$ is an algebraic number.
\end{cor}
\begin{proof}
Every algebraic power series also satisfies a linear differential equation with
polynomial coefficients (see \cite{\Stanley} for many basic facts about D-finite
series). It is known \cite{\Lips} that the constant term of a D-finite series of
two variables is a D-finite series of a single variable. Substituting an
algebraic series into a D-finite series gives another D-finite series, and so
transforming from $[q^0]G(z;q)$ to the cogrowth series (which is done by
substituting a rational function) yields another D-finite series.

Finally, if a function satisfies a linear differential equation, then its
singularities must correspond to zeros of the coefficient of the highest order
derivative. Since the cogrowth series is D-finite, its singularities must be
the zeros of the polynomial coefficient of the highest order derivative.
\end{proof}
While the results used to prove the above corollary guarantee the existence of
such differential equations, they do not give recipes for determining them.
There has been a small industry in developing algorithms to do exactly this
task (and many other operations on D-finite series) --- for example
work by Zeilberger, Chyzak and others. Here we have used recent algorithms
developed by Chen, Kauers and Singer \cite{\Kauers}, and we are grateful for
Manuel Kauers' help in the application of these tools.

Applying the algorithms described in \cite{\Kauers} to the generating function
$G(z;q)$ for $\mathrm{BS}(2,2)$ which is the solution of equation~\Ref{eqn bs22 gzq} we
found a $6^\mathrm{th}$ order linear differential equation satisfied by
$[q^0]G(z;q)$. Unfortunately the polynomial coefficients of this equation have
degrees up to 47. We were also able to guess slightly more appealing equations
of higher order with lower degree coefficients, but all are too large to list
here. 

For $\mathrm{BS}(3,3)$ and $\mathrm{BS}(4,4)$ we obtain the following equations for $G(z;q)$
(where $Q=q+\bar{q}$)
\begin{multline}
1
+4zQG
+(6Q^2z^2-z^2-1)G^2
+2z(Qz+1)(Q^2z-Q+2z)G^3\\
+z^2(1-Q)(1+Q)(Qz+2z-1)(Qz-2z-1)G^4 = 0
\end{multline}
and
\begin{multline}
1
+5GQz
+(10Q^2z^2-2z^2-1)G^2
+z(10Q^3z^2-6Qz^2-3Q+4z)G^3\\
+z^2(3Q^4z^2+2Q^2z^2-3Q^2+8Qz-8z^2+2)G^4\\
-z^3Q(Q^2-2)(Qz+2z-1)(Qz-2z-1)G^5 =0
\end{multline}
Again applying the same methods, we found an ODE of order 8 with coefficients of
degree up to 105 for $\mathrm{BS}(3,3)$ and for $\mathrm{BS}(4,4)$ it is order 10 with
coefficients of degree up to 154. Using clever guessing techniques
(see \cite{\KauersGuess} for a description) Kauers also found DEs for
$N=5,\dots,10$. For $\mathrm{BS}(5,5)$ the DE is order 12 with coefficients of degree up
to 301. While that of $\mathrm{BS}(10,10)$ took about 50 days of computer time to guess
and is 22nd order with coefficients of degree up to 1153; when written in text
file is over 6 Mb! We note that the ODEs found for $N=2,3,4$ have been proved,
but it is beyond current techniques\footnote{While there is no theoretical
barrier, the time taken by the computations seems to grow quickly with $N$ and
exceed the available time.} to prove those found for higher $N$.

Clearly this approach is not a practical means to study the cogrowth for
larger $N$ --- though one can generate series expansions quite quickly
using a computer. We are able to determine the radius of convergence of
$[q^0]G(z;q)$ for much higher $N$ via the following lemma.
\begin{lemma}
\label{lem gz1 gz0}
For $\mathrm{BS}(N,N)$, the generating functions $G(z;1)$ and $[q^0]G(z;q)$ have the same radius of
convergence.
\end{lemma}
\begin{proof}
We start with some notation. Write
\begin{align}
  G(z;q) &= \sum_{n=0}^\infty \sum_{k=-n}^n g_{n,k} z^n q^k 
  & G(z;1) &= \sum_{n=0}^\infty g_n z^n
\end{align}
Note that we have $g_{n,-k} = g_{n,k}$ and that $g_{n,k}=0$ for $|k|>n$. Write
$\limsup g_n^{1/n} = \mu$ and $\limsup g_{n,0}^{1/n} = \mu_0$. Since all the
$g_{n,k}$ are non-negative, we immediately have $\mu \geq \mu_0$.

To prove the reverse inequality we use a ``most popular'' argument that is
commonly used in statistical mechanics to prove equalities of free-energies
(see \cite{\Hammersley} for example). Fix $n$, then there exists $k^*$
(depending on $n$) so that $g_{n,k^*} \geq g_{n,k}$ --- the number $k^*$ is the
``most popular'' $a$-exponent in all the trivial words of length $n$
contributing to the generating function $G$. We have 
\begin{align}
  \label{eqn gnkstar}
  g_{n,k^*} \leq g_n \leq (2n+1) g_{n,k^*}
\end{align}
And hence $\limsup g_{n,k^*}^{1/n} = \mu$. Note that numerical experiments show
that $k^* = 0$ --- the distribution is tightly peaked around 0.

Keeping $n$ fixed, consider a word that contributes to $g_{n,k^*}$ and another
that contributes to $g_{n,-k^*}$. Concatenating them together gives a word that
contributes to $g_{2n,0}$. So considering all possible concatenation of $M$
such pairs of words gives the following inequality
\begin{align}
  g_{n,k^*}^M g_{n,-k^*}^M  = g_{n,k^*}^{2M} & \leq g_{2Mn,0}
\intertext{Raise both sides to the power $\nicefrac{1}{2nM}$ and let $M \to
\infty$ gives} 
  g_{n,k^*}^{1/n} &\leq \mu_0
\end{align}
Letting $n \to \infty$ then shows that $\mu \leq \mu_0$.
\end{proof}
We have observed that the statement of the lemma appears to hold for  Baumslag-Solitar
groups $\mathrm{BS}(M,N)$ for $M\neq N$ also, however the above proof breaks down in the general case as the number
of summands in equation~\Ref{eqn gnkstar} grows exponentially with $n$ rather
than linearly.

Combining Proposition~\ref{prop NN GL eqn} and the above lemma we can
establish the growth rates of trivial words $\mu$ and the corresponding
cogrowths $\lambda$ for the first few values of $N$ (see Table~\ref{tab mu NN}).
Unfortunately we have not been able to find a general form for these
numbers. Some simple numerical analysis of these numbers suggests that the
growth rate approaches $\sqrt{12}$ exponentially with increasing $N$. This
finding agrees with work of Guyot and Stalder \cite{\GStald}, discussed below,
who examined the limit of the {\em marked} groups $\mathrm{BS}(M,N)$ as $M,N\rightarrow
\infty$, and found that the groups tend towards the free group on two letters,
which has an asymptotic cogrowth rate of  $\sqrt{12}$.

\begin{table}
\begin{center}
\begin{tabular}{|c|c||c|}
\hline
$N$ & $\mu$ & $\lambda$ \\
\hline
1& 4 & 3 \\
2& 3.792765039 & 2.668565568 \\
3& 3.647639445 & 2.395062561 \\
4& 3.569497357 & 2.215245886 \\
5& 3.525816111 & 2.091305394 \\
6 & 3.500607636 & 2.002421757\\
7 &  3.485775158 & 1.936941986 \\
8 & 3.476962757 & 1.887871818 \\
9 & 3.471710431 & 1.850717434 \\
10 & 3.468586539 & 1.822458708 \\
\hline
\end{tabular}
\end{center}
\caption{The growth rate $\mu$ of trivial words in $\mathrm{BS}(N,N)$ and the
corresponding cogrowth $\lambda$. Note that $\mu$ and $\lambda$ are related by 
$\mu = \lambda + 3/\lambda$, and that the growth rate of trivial words in the
free group on 2 generators is $\sqrt{12} = 3.464101615$.}
\label{tab mu NN}
\end{table}

We remark that for $\mathrm{BS}(1,1)\cong \mathbb Z^2$ the number of trivial words is known exactly and
hence so is the dominant asymptotic form
\begin{align*}
  g_{n,0} &= \binom{n}{n/2}^2 \sim \frac{2}{\pi n} \cdot 4^n 
& \mbox{for even $n$}.
\end{align*}
In the case of $N=2,3,4,5$ we can show from the differential equations found
above that
\begin{align*}
  g_{n,0} & \sim A_N \mu_N^n n^{-2}
& \mbox{for even $n$}
\end{align*}
where $\mu_N$ is given in the previous corollary and we have estimated the
amplitudes to be
\begin{align*}
  A_2 &= 12.47372070225776\dots &
  A_3 &= 10.81007294255599\dots \\
  A_4 &= 12.14125535742978\dots &
  A_5 &= 14.73149478993552\dots.
\end{align*}
Unfortunately we have not been able to identify these constants, but these
observations lead to the following conjecture.

\begin{conj}
The number of trivial words in $\mathrm{BS}(N,N)$ grows as
\begin{align*}
  g_{n,0} \sim A_N \mu_N^n n^{-2}
\end{align*}
for $N \geq 2$.
\end{conj}

\subsection{Continued fractions and $\mathrm{BS}(1,M)$}
When we set $N=1$ cancellations occur and the equation for $L$ becomes a
$q$-deformation of a Catalan generating function:
\begin{align}
 L(z;q) &= 1 + z(q+\bar{q})L(z;q) + z^2 L(z;q) L(z;q^M) 
  = \frac{1}{1 - z(q+\bar{q})-L(z;q^M)},\nn
 L(z;1) &= \frac{1-2z-\sqrt{1-4z}}{2z^2}.
\end{align}
Setting $q=1$ into the first equation reduces it to algebraic and it is readily
solved to give $L(z;1)$ which is the generating function of the Catalan
numbers. Thus $L(z;q)$ is a $q$-deformation of the Catalan numbers and
rearranging the first equation shows that $L(z;q)$ has a simple continued
fraction expansion.
\begin{align}
  L(z;q) &= \cfrac{1}{
  1-z(q+q^{-1})- \cfrac{z^2}{
  1-z(q^{M}+q^{-M})- \cfrac{z^2}{
  1-z(q^{M^2}+q^{-M^2})- 
  \cdots
  }
  }
  }.
\end{align}
Such continued fraction forms are well known and understood in Catalan objects
(see \cite{\FlajoletContFrac} for example). Unfortunately the equation for $K$
does not simplify:
\begin{align}
  K 
  &= 1 + z(q+\bar{q})K + z^2 L(z;q^M) \cdot \left[K - \Phi_{M,M}\circ K\right]
  + z^2 K \cdot \left[ \Phi_{M,1} \circ K \right].
\end{align}
Though as noted above $K(z;1) = L(z;1)$ and so we expect $K(z;q)$ to be a
different $q$-deformation of the Catalan numbers. For $G$ we have made even
less progress and we have not found $G(z;1)$, let alone $G(z;q)$, in closed
form. Because of the $q$-deformed nature of $L(z;q)$ we conjecture the following
\begin{conj}
 For Baumslag-Solitar groups $\mathrm{BS}(1,M)$ with $M>1$, the generating functions
$G(z;q)$ and $[q^0]G(z;q)$ are not D-finite.
\end{conj}
Since any path that contributes to $K$ or $L$ must also contribute to $G$, it
follows that the radius of convergence of $G(z;1)$ is at most $\nicefrac{1}{4}$
--- and of course cannot be any smaller. Since the groups $\mathrm{BS}(1,N)$ are all
amenable, we know that $g_{n,0} \sim 4^n$. We have been unable to prove any more
precise details of the asymptotic form, though it is not unreasonable to expect
that 
\begin{align}
  g_{n,0} \sim A 4^n n^{-\gamma_M}.
\end{align}
While we have been able to generate the first 50-60 terms of the series for $M
\leq 5$ by iterating the equations, the series are quite badly behaved and
we have been unable to produce reasonable estimates of $\gamma_M$.

\subsection{When $N \neq M$}
When $N \neq M$, we expect that the operators $\Phi_{N,M}$ and $\Phi_{M,N}$ in
the equations satisfied by $G,K,L$ give rise to $q$-deformations similar
to those observed above. In light of this, we extend our previous
conjecture:
\addtocounter{conj}{-1}
\begin{conj}[Extended from the above]
For Baumslag-Solitar groups $\mathrm{BS}(N,M)$ the generating functions $G(z;q)$ and
$[q^0]G(z;q)$ are only D-finite when $N = M$.
\end{conj}

In spite of the absence of D-finite recurrences, we can still use the equations
above to determine the first few terms of the cogrowth series. The resulting
algorithm to compute the first $n$ terms of the series requires time and memory
that are exponential in $n$. The coefficient of $z^n$ is a Laurent polynomial
whose degree is exponential in $n$ and this exponential growth becomes worse as
$\max\{N/M, M/N\}$ becomes larger. In spite of this, iteration of these
equations to obtain the cogrowth series is exponentially faster than more naive
methods based on say a simple backtracking exploration of the Cayley graph, or
iteration of the corresponding adjacency matrix.

The time and memory requirements can be further improved since we are primarily
interested in the constant term; this means that we do not need to keep high
powers of $q$. More precisely if we wish to compute the series to $O(z^n)$, then
we only need to retain those powers of $q$ that will contribute to
$[q^0z^n]G(z;q)$. We must compute the coefficients of $z^k$ for $k \leq n/2$
exactly, but we can ``trim'' subsequent coefficients --- the degree of
$z^{n/2+k}$ needs only be that of $z^{n/2+k}$. 

Simple \texttt{c++} code using \texttt{cln}\footnote{An open source \texttt{c++}
library for computations with large integers. At time of writing it is
available from \texttt{http://www.ginac.de/CLN/} }
running on a moderate desktop allowed us to generate about the first 50 terms
of $[q^0]G(z;q)$ for $\mathrm{BS}(1,5)$ while over 300 terms for $\mathrm{BS}(4,5)$
were obtained. The series lengths for the other (with $N<M \leq 5$) 
ranged between these extremes. 
We have estimated the growth
rate of trivial words using differential approximants --- see Table~\ref{tab mu
NM}. Again like the $N=1$ case, we find the series to be very badly behaved
(except when $N = M$) and we have only obtained quite rough estimates. 

\begin{table}
\begin{center}
 \begin{tabular}{|l||c|c|c|c|c|}
\hline
 \backslashbox{$N$}{$M$} & 1 & 2 & 3 & 4 &5 \\
\hline\hline
1 & 4 & 4 & 4 & 4 & 4 \\ 
\hline
2 & $\circ$ & 3.792765039 & 3.724$\star$ & 3.701$\star$ &3.676$\star$ \\
\hline
3 & $\circ$ & $\circ$ & 3.647639445  & 3.604$\star$& 3.585$\star$\\
\hline
4 & $\circ$ & $\circ$ & $\circ$ & 3.569497357 & 3.538$\star$\\
\hline
5 & $\circ$ & $\circ$ & $\circ$ & $\circ$ & 3.525816111\\
\hline
 \end{tabular}
\end{center}
\caption{Exact and estimated growth rates of trivial words for Baumslag-Solitar
groups~$\mathrm{BS}(N,M)$. The estimated growth rates are denoted with a $\star$ and
we expect that the error lies in the last decimal place --- due to the
difficulty of obtaining estimates, they should be considered to be quite
rough.}
\label{tab mu NM}
\end{table}

\subsection{The limit as $N,M \to \infty$}
Beautiful work of Luc Guyot and  Yves Stalder \cite{\GStald} demonstrates that in the limit as
$N,M \to \infty$ the (marked) group $\mathrm{BS}(N,M)$ becomes the free group on 2
generators. We note that we can observe this free group behaviour in the
functional equations we have obtained.

In particular as $N,M \to \infty$, the operators $\Phi_{N,N}, \Phi_{M,M},
\Phi_{N,M}$ and $\Phi_{M,N}$ become the constant-term operators. So
in this limit the equations for $K$ and $L$ from Proposition~\ref{prop GLK eqn}
become
\begin{align}
 L &=  1 + z(q+\bar{q})L + z^2 L\cdot \left[ L_0 + K_0 \right] - z^2 K_0 L_0,
\nn
 K & = 1 + z(q+\bar{q})K + z^2 K\cdot\left[ K_0 + L_0 \right] - z^2 K_0 L_0,
\end{align}
where $K_0(z) = [q^0]K(z;q)$ and $L_0(z) = [q^0]L(z;q)$. Clearly $K(z;q) =
L(z;q)$ and so with a little rearranging
\begin{align}
  L(z;q) &= \frac{1-z^2 L_0(z)^2}{1-z(q+\bar{q}) - 2z^2L_0(z)}
  = \frac{1-z^2 L_0^2}{1-2z^2L_0} 
  \sum_{n \geq 0} \left( \frac{z(q+\bar{q})}{1-2z^2L_0} \right)^n.
\end{align}
Taking the constant term of both sides then gives
\begin{align}
  L_0 &= \frac{1-z^2 L_0^2}{1-2z^2L_0} 
  \sum_{n \geq 0} \binom{2n}{n} \left( \frac{z}{1-2z^2L_0} \right)^{2n} 
  = \frac{1-z^2 L_0^2}{1-2z^2L_0}
  \left[1-4\left(\frac{z}{1-2z^2L_0}\right)^2\right]^{-1/2}
\end{align}
Simplifying this last expression further gives
$(3z^2L0^2-L0+1)(z^2L0^2-L0-1)=0$. The only positive term power series solution
of this gives $L_0$ and a similar exercise gives $[q^0]G(z;q)$:
\begin{align}
  L_0 &= \frac{1-\sqrt{1-12z^2}}{6z^2} & 
  [q^0] G &= \frac{3}{1+2\sqrt{1-12z^2}}
\end{align}
The expression for $[q^0]G$ is the number of trivial words in the free group on
2 generators.

\section{Analysis of random sampling data}
\label{sec:analysis}
\subsection{Preliminaries}
Using our multiple Markov chain Monte Carlo algorithm we have sampled trivial
words from the following groups:
\begin{itemize}
 \item Baumslag-Solitar groups $\mathrm{BS}(N,M)$ with 
$$ (N,M)  =    (1,1), (1,2), (1,3),
(2,2), (2,3),   (3,3), (3,5).$$
\item  The basilica group has presentation
\begin{equation}
G = \langle a,b,\,|\,\hbox{$ \left[ a^n,[a^n,b^n]\right]$ and 
                                        $ \left[ b^n,[b^n,a^{2n}] \right]$}
                                           ,\,\hbox{$n$ a power of $2$} \rangle
\end{equation}
and embeds in the finitely presented group \cite{Grigorchuk02}
\begin{equation}
\W{G} = \langle a,t \,|\, a^{t^2}=a^2,\,\left[ \left[ [a,t^{-1}] ,a\right],a\right] = 1 \rangle .
\end{equation}
The groups $G$ and $\W{G}$ are both amenable   \cite{\BartV}.  

We examined two presentations of $\W{G}$:  The first is obtained from the above
by putting $b=[a,t^{-1}]$, and the second by putting $b=a^t$.  Simplification gives
the representations
\begin{align}
\W{G} &=  \langle a,b,t \,|\, \hbox{$[a,t^{-1}]=b$
                                                       ,\,$a^{t^2}=aa$
                                                       ,\,$\left[ [b,a],a \right]=1 $}\rangle , 
\label{eqnA4.4} \\
\W{G} &=  \langle a,b,t \,|\, \hbox{$a^t=b$
                                                       ,\,$b^t=a^2$
                                                       ,\,$b^{-1}aba^{-1}b^{-1}a^{-1}ba=1 $}\rangle . 
\label{eqnA4.5}
\end{align}
\item Other groups for which the cogrowth series is known:
\begin{align}\label{presKouksov}
 K_1 &= \langle a,b \,|\, a^2=b^3=1 \rangle, \nn
K_2 & =  \langle a,b \,|\, a^3=b^3=1 \rangle, \nn
K_3 & =  \langle a,b,c \,|\, a^2=b^2=c^2=1 \rangle.
\end{align}

 \item Thompson's group $F$ with the following 3 presentations
\begin{align}
  \langle a,b &\,|\, [a\bbar,\abar b a], [a\bbar,a^{-2} b a^2] \rangle,  \label{eqnF1}\\
  \langle a,b,c,d &\,|\, c=\abar b a, d=\abar c a, [a\bbar,c], [a\bbar, d]
\rangle, \label{eqnF2}\\
  \langle a,b,c,d,e & \,|\, c=\abar b a, d=\abar c a, e=a\bbar, [e,c], [e,d]
\rangle. \label{eqnF3}
\end{align}
Note that the generators $a,b,c,d$ above are often called $x_0,x_1,x_2,x_3$
respectively in Thompson's group literature. We have use some simple {\em Tietze
transformations} (see \cite{\LS} p. 89) to obtain the second and third
presentations from the first (standard) finite  presentation of $F$.
\end{itemize}
The exact solutions for $\mathrm{BS}(1,1)\cong \mathbb Z^2, \mathrm{BS}(2,2)$ and $\mathrm{BS}(3,3)$ are
described above, and for the other Baumslag-Solitar groups we have computed
series expansions. For the last three groups, the cogrowth series  were found by
Kouksov \cite{\Kouksov} and are (respectively)
\begin{align}
C(t) &= \frac{(1+t)\left([0,-1,1,-8,3,-9] +(2-t+6t^2)\sqrt{
[1,-2,1,-6,-8,-18,9,-54,81]}\right)}
{2(1-3t)(1+3t^2)(1+3t+3t^2)(1-t+3t^2)}\nn
%
%
C(t) &= \frac{(1+t)(-t+\sqrt{1-2t-t^2-6t^3+9t^4})}{(1-3t)(1+2t+3t^2)}\nn
%
C(t) &= \frac{-1-5t^2+3\sqrt{1-22t^2+25t^4}}{2(1-25t^2)}
\end{align}
where $[c_0,c_1,\dots,c_n] = c_0+c_1t+\cdots+c_nt^n$. 

In each case we have obtained estimates of the mean length of freely reduced
words as a function of $\beta$. More precisely, for each group we estimated
\begin{align}
  \Ex( n^k )(\beta)
&= \frac{\sum_\omega |w|^k (|w|+1)^{1+\alpha} \beta^{|w|}}
  {\sum_\omega (|w|+1)^{1+\alpha} \beta^{|w|}}
\label{eqn mlk}
\end{align}
for $k=\pm 1, \pm 2$ and a range of different $\alpha$ values and where the sum
is over all non-empty freely reduced trivial words. These expectations are
dependent on $\alpha$, but one can use Equation~\Ref{eqn rat est} to form
$\alpha$-independent estimates of the canonical expectations. Given a knowledge
of the cogrowth series we can quickly compute these same means to any desired
precision, since we can also write
\begin{align}
  \Ex( n^k )(\beta) &= \sum_{n \geq 0} 
\frac{n^k (n+1)^{1+\alpha} p_n \beta^n}{(n+1)^{1+\alpha} p_n \beta^n}
\end{align}
where $p_n$ is the number of freely reduced words of length $n$. Note that as
$\alpha$ is increased, the samples are biased towards longer words. This
expression is convergent for $\beta$ below the reciprocal of the cogrowth (being
the critical point of the associated generating function) and divergent above
it. The convergence \emph{at} the critical point depends on the precise details
of the asymptotics of $p_n$ and will be effected by $\alpha$. This then points
to a simple way to test for amenability.

\begin{prop}
 If the mean length of sampled words from a group on $k$ generators is
finite for $\beta$ slightly above $\beta_c=(2k-1)^{-1}$ then the group is
not amenable.
\end{prop}


Note that the amenability or non-amenability of Thompson's group $F$ is an open question, and one that has received an intense amount of interest and study.

\subsection{Amenable groups}
We studied the groups $\mathbb{Z}^2 \cong \mathrm{BS}(1,1), \mathrm{BS}(1,2)$ and $\mathrm{BS}(1,3)$. The
cogrowth series for $\mathbb{Z}^2$ is known exactly, while we relied on our
series expansions to compute statistics for the other two groups ---
Figure~\ref{fig:bs111213 plots} shows the plots of the mean length as a
function of $\beta$.

\begin{figure}[h]
  \centering
  \subfloat[$\mathrm{BS}(1,1)$ sampled with $\alpha = 1$.]{
    \includegraphics[width=0.5\textwidth]{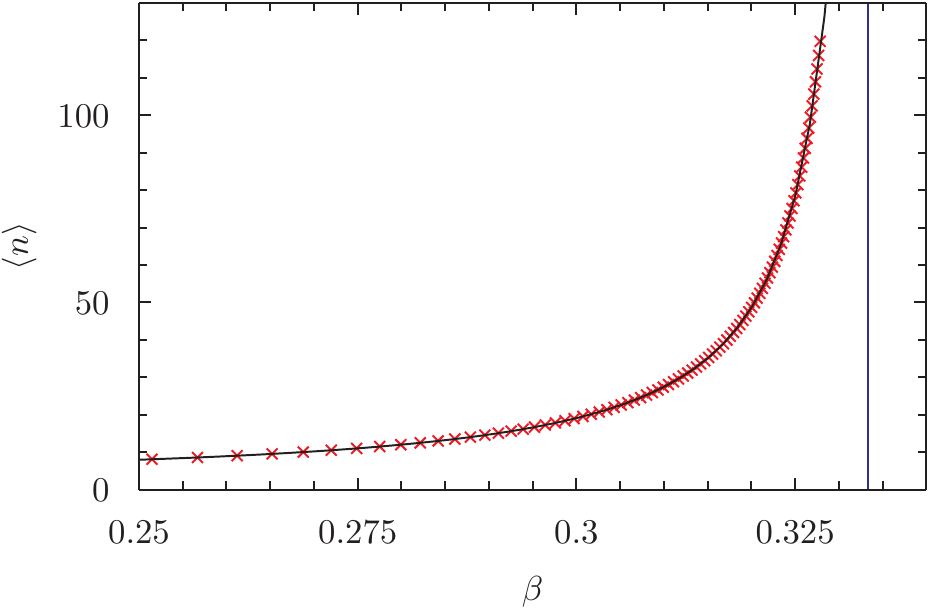}}
  \subfloat[$\mathrm{BS}(1,2)$ sampled with $\alpha = 1$.]{
    \includegraphics[width=0.5\textwidth]{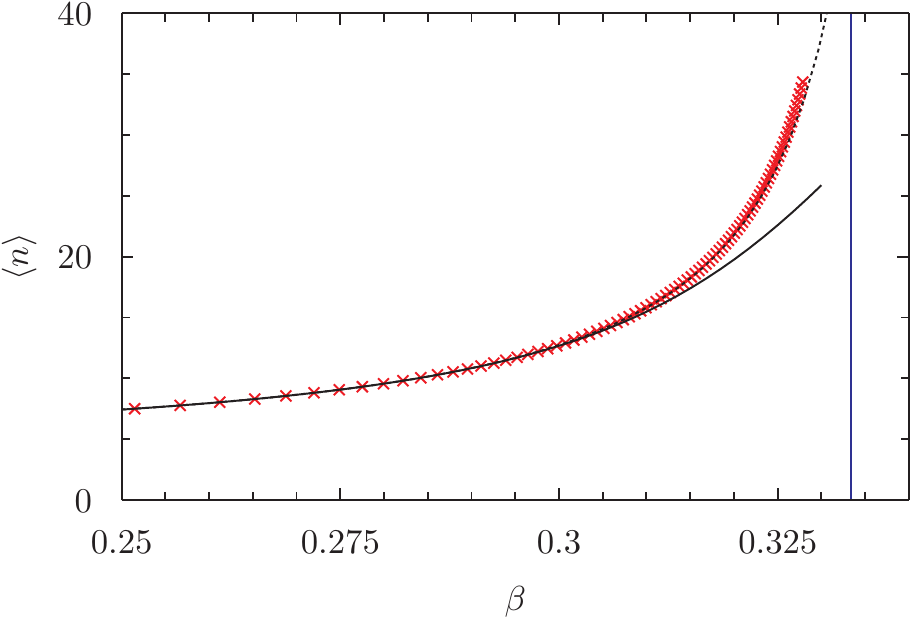}}\\
  \subfloat[$\mathrm{BS}(1,3)$ sampled with $\alpha = 2$.]{
    \includegraphics[width=0.5\textwidth]{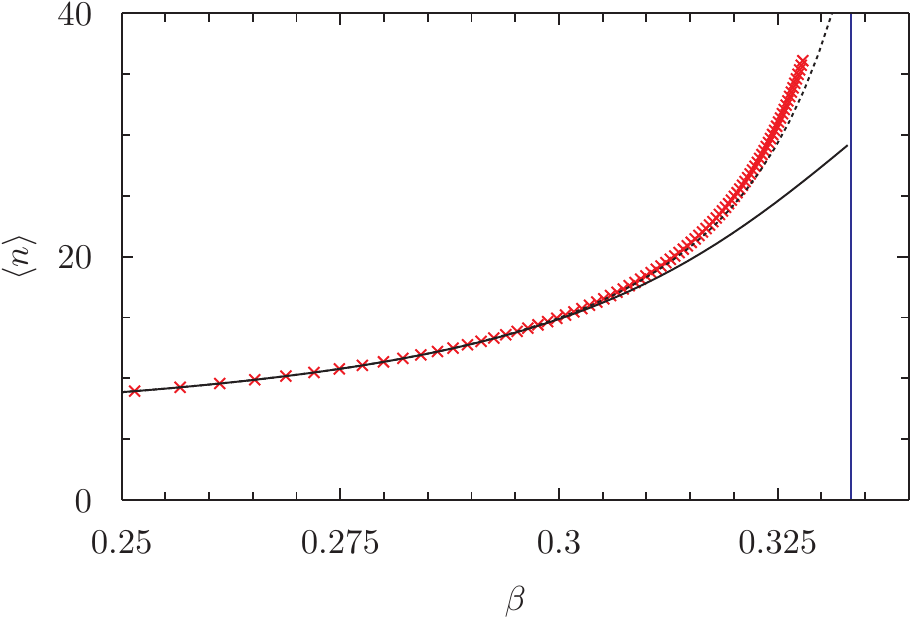}}
  \caption{Mean length of freely reduced trivial words in Baumslag-Solitar
groups $\mathrm{BS}(1,1), \mathrm{BS}(1,2)$ and $\mathrm{BS}(1,3)$ at different values of~$\beta$ --- see
equation~\Ref{eqn mlk} with $k=1$ and $\alpha$ as indicated. The sampled points
are indicated with impulses, while the exact values are given by the black line.
There is excellent agreement in the case of $\mathrm{BS}(1,1)$, but a systematic error
for $\mathrm{BS}(1,2)$ and $\mathrm{BS}(1,3)$ coming from the modest length of
the exact series. The dotted lines in these two cases indicate mean length data
generated from longer but approximate series.}
  \label{fig:bs111213 plots}
\end{figure}

\begin{figure}[h]
  \centering
  \subfloat[$\mathrm{BS}(1,1)$ sampled with $\alpha = 0,1,2,3$.]{
    \includegraphics[width=0.5\textwidth]{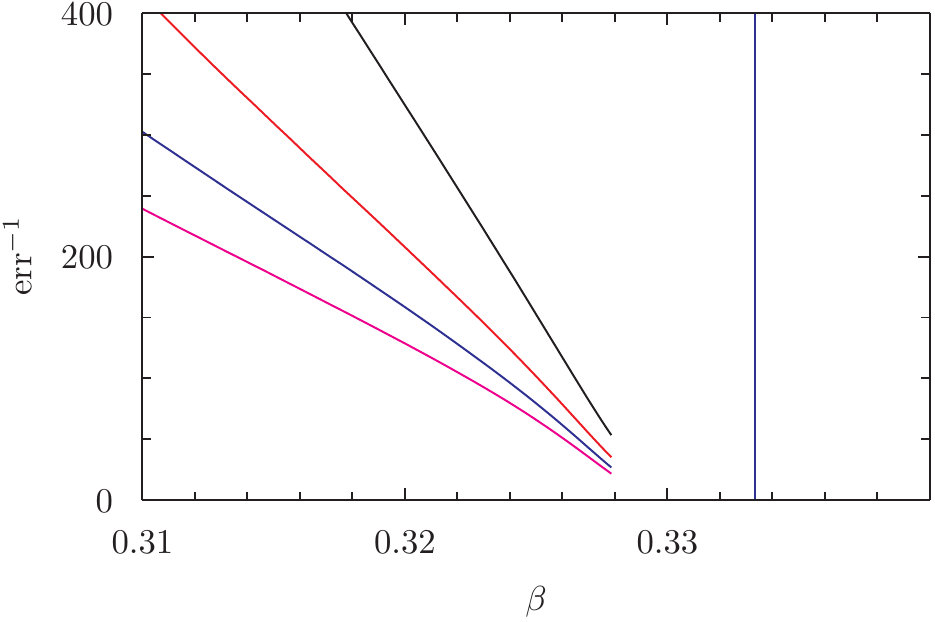}}
  \subfloat[$\mathrm{BS}(1,2)$ sampled with $\alpha = 0,1,2,3$.]{
    \includegraphics[width=0.5\textwidth]{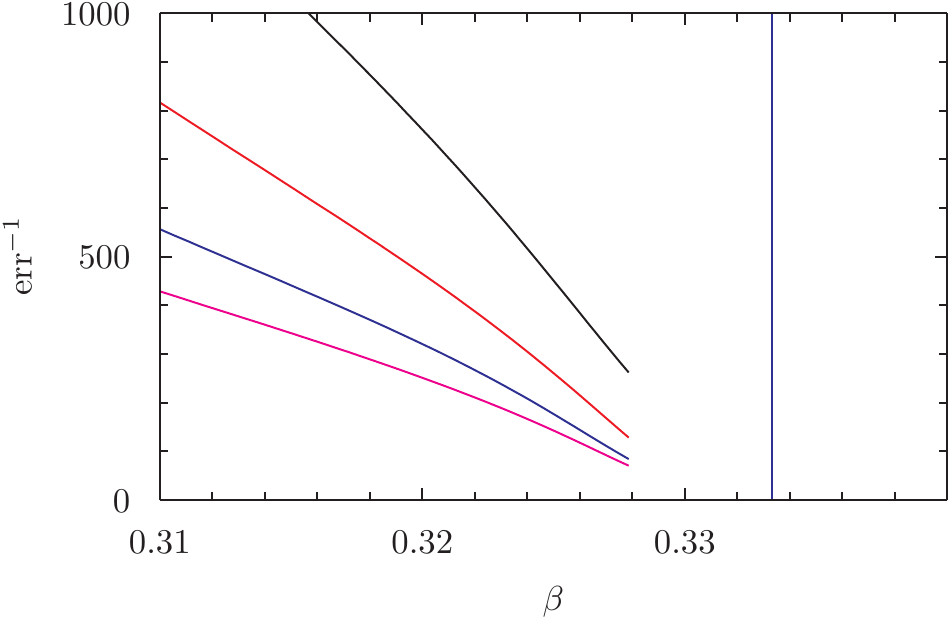}}\\
  \subfloat[$\mathrm{BS}(1,3)$ sampled with $\alpha = 0,1,2,3$.]{
    \includegraphics[width=0.5\textwidth]{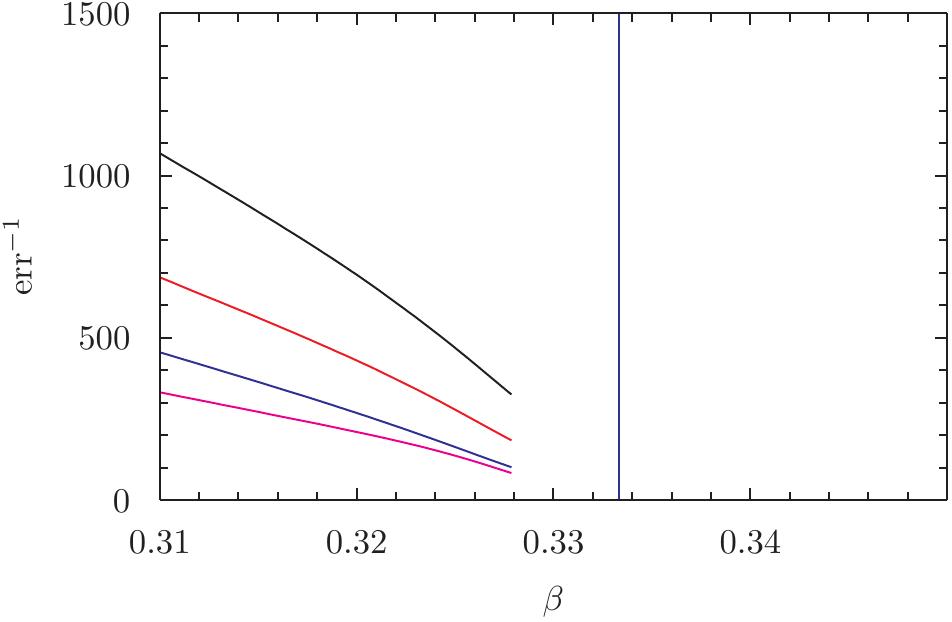}}
  \caption{Plots of the reciprocal of estimated standard error in the mean
length vs beta for $\alpha=0,1,2,3$ anti-clockwise from the top. We expect
that as $\beta$ approaches its critical value, that the standard error will
diverge. We see that if we extrapolate the curves then they cross the $x$-axis
at $\beta = 0.330\pm0.002, 0.332\pm0.002$ and $\beta=0.332\pm0.002$
respectively --- thus these extrapolations give good estimates of the critical
value of $\beta$.}
  \label{fig:bs111213 err plots}
\end{figure}

In the case of $\mathrm{BS}(1,1)\cong \mathbb{Z}^2$ we see excellent agreement between the numerical
estimates generated by our algorithm and the mean length computed from the exact
cogrowth series. For $\mathrm{BS}(1,2)$ and $\mathrm{BS}(1,3)$ we see good agreement for low
$\beta$ between our numerical data and mean length computed from the exact
cogrowth series. However at larger values of $\beta$ it appears that the
cogrowth series systematically underestimates the mean length, compared
to the numerical Monte Carlo data. This is, in fact,
due to the modest length of the cogrowth series used to compute mean lengths.
For $\mathrm{BS}(1,2)$ and $\mathrm{BS}(1,3)$ we were only able to obtain series of length 60 and
56 respectively due to memory constraints. Given longer series we expect much
better agreement.

One can, for example, compute longer ``approximate'' cogrowth series by
ignoring small terms. When iterating the functional equations given in
Proposition~\ref{prop GLK eqn} one can form reasonable approximations by
discarding coefficients $g_{n,k}$ which are small compared to nearby
coefficients.\footnote{Rather than iterating the equations for $G(z;q)$ and
then transforming the result to get an approximate cogrowth series, we found
that our approximation procedure worked best if we iterated the slightly more
complicated equations for the cogrowth series directly --- see text following
Proposition~\ref{prop GLK eqn} for a description of those equations.} %
More precisely we found that if we discard $g_{n,k}$ when $2^{12}
\cdot g_{n,k} < \sum_k g_{n,k}$, then we obtain good approximations of the
cogrowth series. This means that only the large central coefficients are kept
and far less memory used. This made it feasible to approximate the cogrowth
series out to around 200 or 300 terms. Of course, the results of these
approximation should only be considered a rough guide as we have not bounded the
size of any resulting errors. That being said, we see very good agreement
between these approximations and our numerical data.

As noted above, we had great difficulty fitting the series data for $\mathrm{BS}(1,2)$ and 
$\mathrm{BS}(1,3)$. We believe that this is due to the presence of complicated confluent
corrections (likely logarithmic terms). Similar corrections also appear to be
present in the mean-length data for these groups and we were unable to find
convincing or consistent fits to any reasonable functional forms. We did,
however, find that the estimated standard error was a good indicator of the
location of the singularity: The standard error will diverge as $\beta$
approaches the critical value of $\nicefrac{1}{3}$. We found that linear or
quadratic least squares fits of the reciprocal of the error, and
finding their $x$-intercept gave consistent, though perhaps slightly low,
estimates of the location of the singularity. See Figure~\ref{fig:bs111213 err
plots}. The extrapolations give estimates $\beta = 0.330\pm0.0002,
0.332\pm0.002$ and $\beta=0.332\pm0.002$ for $\mathrm{BS}(1,1), \mathrm{BS}(1,2)$ and $\mathrm{BS}(1,3)$
respectively.

Error bars above were determined by estimating a systematic error in our data.
The systematic error was determined by considering the spread of estimates due
to our choices of the parameter $\alpha$, the number of data points in the fits,
and the chosen functional form for extrapolating the data. We believe that our
results give a good indication of quality of the  estimates, though we are
reluctant to express them as firm confidence intervals.  The same general
approach to the data for the other groups are followed below. 

The HNN-extension of the basilica group were similarly submitted to Monte
Carlo simulation by using the representations \Ref{eqnA4.4} and \Ref{eqnA4.5}.
The canonical expected length of the words,
$\langle |w|\rangle$, were computed using the ratio
estimator \Ref{eqn rat est}, and turned out to be remarkably insensitive to 
the parameter $\beta$ (see Figure \ref{fig:figureBasilica}).  
This made this group more challenging from a numerical
perspective than the Baumslag-Solitar groups discussed above.
Putting $\alpha=5$ finally gave acceptable results:  The sample average
length show a divergence close to the critical point (since this group
is known to be amenable, this is expected to be at $\beta=0.2$).
As in the case of the Baumslag-Solitar groups, the critical value of 
$\beta$ was determined by extrapolating the reciprocal of the error.
Extrapolating the curve corresponding to representation \Ref{eqnA4.4}
gave $\beta_c=0.217$ and for representation \Ref{eqnA4.5},
$\beta_c=0.204$.  Taking the average and using the absolute difference
as a confidence interval gives the estimate $\beta_c = 0.21 \pm 0.01$ to
two digits accuracy.

\begin{figure}[h]
  \centering
  \subfloat[Mean length with $\alpha=5$.]{
   \includegraphics[width=0.5\textwidth]{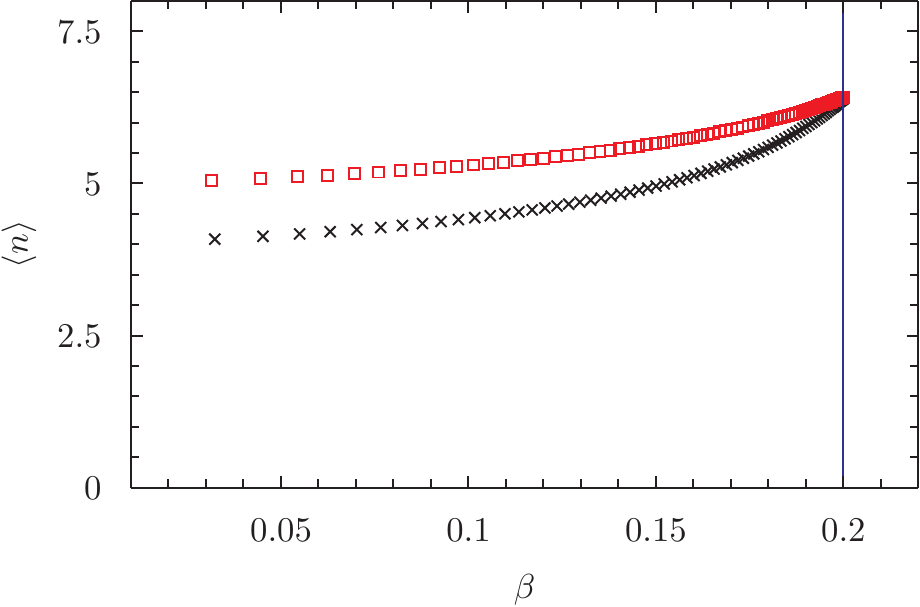}}
 \subfloat[$\hbox{err}^{-1}$ with $\alpha=5$.]{
   \includegraphics[width=0.5\textwidth]{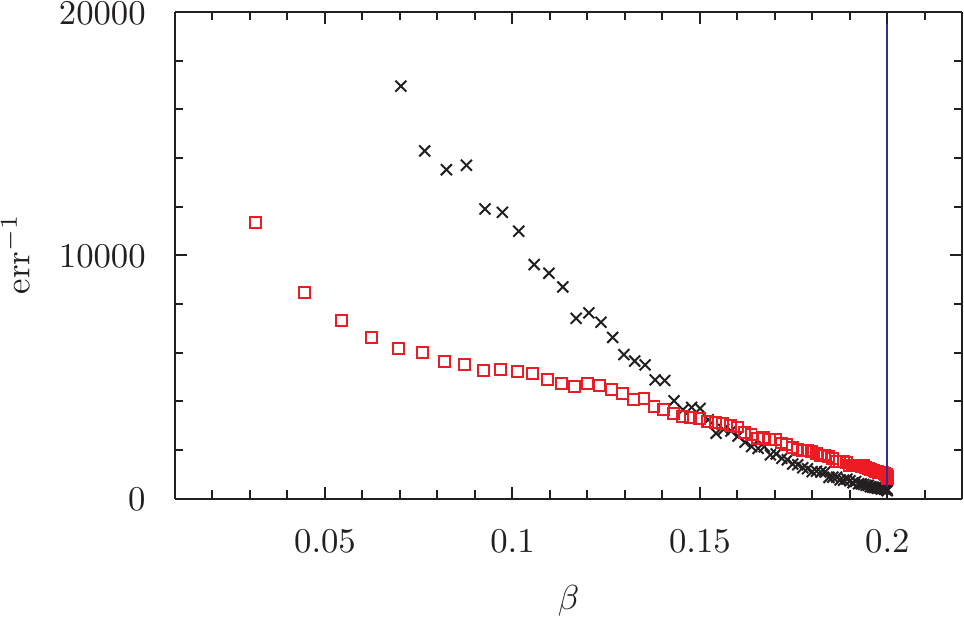}}
 \caption{Numerical data on the HNN-extension of the basilica group.  Data points
indicated by $\square$ corresponds to the representation in equation \Ref{eqnA4.4}
and by $\times$ to the representation in equation \Ref{eqnA4.5}.  In both simulations
$\alpha=5$.  On the left is a plot of the canonical expected length 
$\langle n \rangle$.  These expected lengths are only weakly dependent on
$\beta$.  On the right is the reciprocal error bar on our data.
This demonstrates that the error diverges as $\beta\nearrow 0.20$, consistent
with the fact that these this group is amenable.}
 \label{fig:figureBasilica}
\end{figure}

\subsection{Non-amenable groups}
The groups $\mathrm{BS}(N,M)$ with $(N,M) =$ $ (2,2), (2,3),$ $(3,3), (3,5)$ and 
the groups $K_1,K_2,K_3$ contain a non-abelian free subgroup and so 
are non-amenable.  In the case of the groups $K_1$ and $K_2$ the free 
subgroups are $F((ab),(ab^{-1}))$, and for $K_3$ the free subgroup is 
$F((ab),(ac))$. 

As noted above, the exact cogrowth series is known exactly for Kouksov's
examples and $\mathrm{BS}(2,2), \mathrm{BS}(3,3)$, so we were able to compute the mean length
curves exactly --- see Figures~\ref{fig:kuksov plots} and~\ref{fig:bs2233
plots}. As above, we have estimated the location of the dominant
singularities for all of these groups --- see Figures~\ref{fig:kuksov err plots}
and~\ref{fig:bs2233 err plots}.

Unfortunately we have been unable to solve $\mathrm{BS}(2,3)$ and $\mathrm{BS}(3,5)$, but we used
the recurrences of the previous section to compute the first 100 and 120 terms
(respectively) of their cogrowth series. And as was the case for $\mathrm{BS}(1,2)$ and
$\mathrm{BS}(1,3)$ we also computed an approximation of the cogrowth series
using the same method described above. These are plotted against our Monte
Carlo data in Figures~\ref{fig:bs23 plots} and~\ref{fig:bs2335 err plots}.

In all cases we see strong agreement between our numerical estimates and the
mean length curves computed from series or exact expressions. As was the
case with the amenable groups above, fitting the reciprocal of the
estimated standard error gives quite acceptable estimates of the location of
the dominant singularities and so the cogrowth.

\subsection{Thompson's group}

Finally we come to Thompson's group for which we examine three different
presentations as described above. Repeating the same analysis we used on
the previous groups we see no evidence of a singularity in the mean length at
the amenable values of $\beta$ --- see Figures~\ref{fig:thomp plots}
and~\ref{fig:thomp err plots}. Indeed our estimates of the location of the
dominant singularities are
\begin{align}
 \beta_c &= 0.395\pm0.005, 0.172\pm 0.002 \mbox{ and } 0.134\pm0.004 \mbox{
respectively}.
\end{align}
These give cogrowths of $2.53\pm0.03, 5.81\pm0.07$ and $7.4\pm0.2$, all of
which are well below the amenable values of 3,7 and 9. 

\section{Conclusions}
\label{sec:conclusion}
We have introduced a Markov chain on freely reduced trivial words of any given
finitely presented group. The transitions along the chain are defined in terms
of conjugations by generators and insertions of relations. These moves are
irreducible and satisfy a detailed balance condition; the limiting distribution
of the chain is therefore a stretched Boltzmann distribution over trivial
words. 

In order to validate the algorithm we have implemented it for a range of
finitely presented groups for which the cogrowth series is known exactly. We have
also added to this set of groups by finding recurrences for the cogrowth series
of all Baumslag-Solitar groups. Unfortunately, these recurrences do not have
simple closed-form solutions, but can be iterated to obtain far longer series
than can be found using brute-force methods. In the case of $\mathrm{BS}(N,N)$, the
recurrences simplify significantly and we are able to compute the cogrowth
exactly. For $N=1,\dots,10$ we have found differential equations satisfied by
the cogrowth series which can be used to generate the cogrowth series in
polynomial time.

We see excellent agreement between our mean-length estimates and those computed
exactly for several groups. As a further check on our simulations, two of the
authors independently coded the algorithm and compared the results. We can use
our data to estimate the location of the singularity in the generating function
of freely reduced trivial words. The location of this singularity is the
reciprocal of the cogrowth and so turns out to be an excellent way to predict
the amenability of groups. To test this, we used our algorithm on a range of
different amenable and non-amenable groups. In each case we found that our
numerical estimate of the cogrowth was completely consistent with the known
properties of the groups. In particular, where cogrowth is known exactly, our
numerics agreed. For each non-amenable group, the numerical ``signal'' was
robust --- no evidence of a singularity was seen at the amenable value.

Interestingly, we see 
no evidence that the mean length of
Thompson's group is divergent close to the amenable value; {\em  i.e.} for
2,4 and 5 generator presentations we see no evidence of a singularity at $\beta
= \nicefrac{1}{3}, \nicefrac{1}{7}$ or $\nicefrac{1}{9}$ (respectively). Indeed,
in each case, the mean length appears to be very smooth for $\beta$-values some
reasonable distance above these points. Varying $\alpha$ or examining other
statistics does not result in any substantial change with the result that values
of $\beta$ consistent with amenability are excluded from our estimated error
bars. Overall, our numerical data appears to suggest  that Thompson's
group $F$ is not amenable. However, the question of the amenability of this group have proven to be particularly subtle, so one way to interpret our data is
to say that if $F$ is indeed amenable then it is a highly atypical representative in its class.


As an additional note, we have applied our methods to a finitely generated,
but not finitely presented amenable group --- 
\begin{align}\label{ZwrZ}
\mathbb Z\wr \mathbb Z &= \langle a,b \,|\, [t^iat^{-i},t^jat^{-j}]  \ \forall i,j\in\mathbb Z\rangle.
\end{align}
In this case
the algorithm has to be modified slightly. One can no longer choose relations
uniformly at random, but instead we choose them from distribution $P(R)$
over the relations $[t^iat^{-i},t^jat^{-j}] $. As noted in section 2, this distribution must be 
positive and  one must have $P(R) = P(R^{-1})$. With these conditions 
the algorithm remains ergodic on the space of
trivial words and the stationary distribution is still a 
stretched Boltzmann distribution.
This leaves a great deal of freedom in choosing $P$, and our experiments
indicated that our results were quite independent of $P$ and are consistent
with the amenability of the group.

\section*{Acknowledgements}
The authors thank Manuel Kauers for assistance with establishing the
differential equations described in Section~\ref{sec bsnn}. Similarly we thank
Tony Guttmann for discussions on the analysis of series data. Finally we would
like to thank Sean Cleary and Stu Whittington for many fruitful discussions.

This research was supported by the Australian Research Council (ARC), the
the
Natural Sciences and Engineering Research Council of Canada (NSERC), and the Perimeter Institute for Theoretical
Physics.  Research at Perimeter Institute is supported by the Government
of Canada through Industry Canada and by the Province of Ontario through
the Ministry of Economic Development and Innovation.

\begin{figure}[h]
  \centering
  \subfloat[$\langle a,b|a^2,b^3 \rangle$ sampled with $\alpha = 0$.]{
    \includegraphics[width=0.5\textwidth]{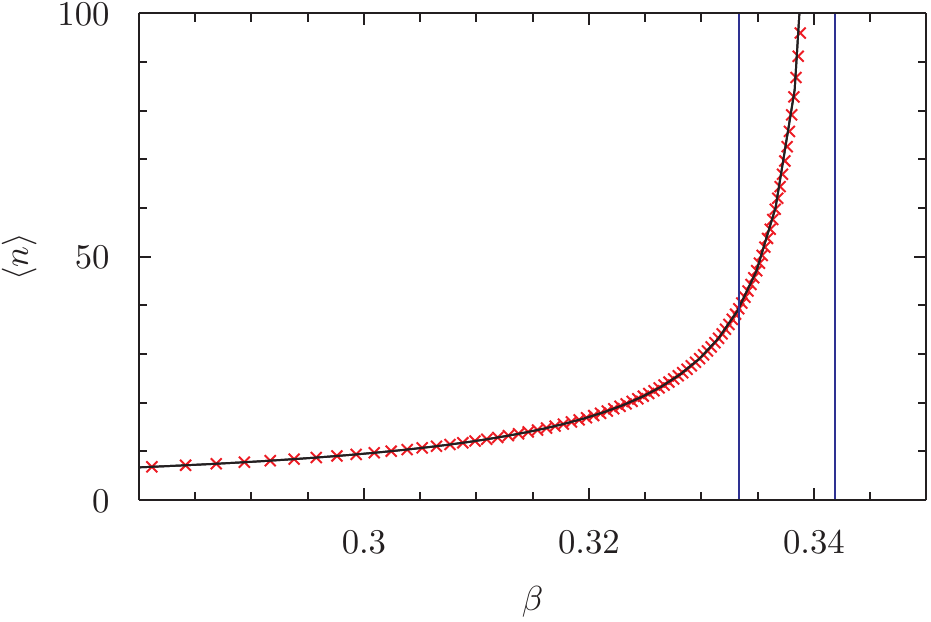}}
  \subfloat[$\langle a,b|a^3,b^3 \rangle$ sampled with $\alpha = 0$.]{
    \includegraphics[width=0.5\textwidth]{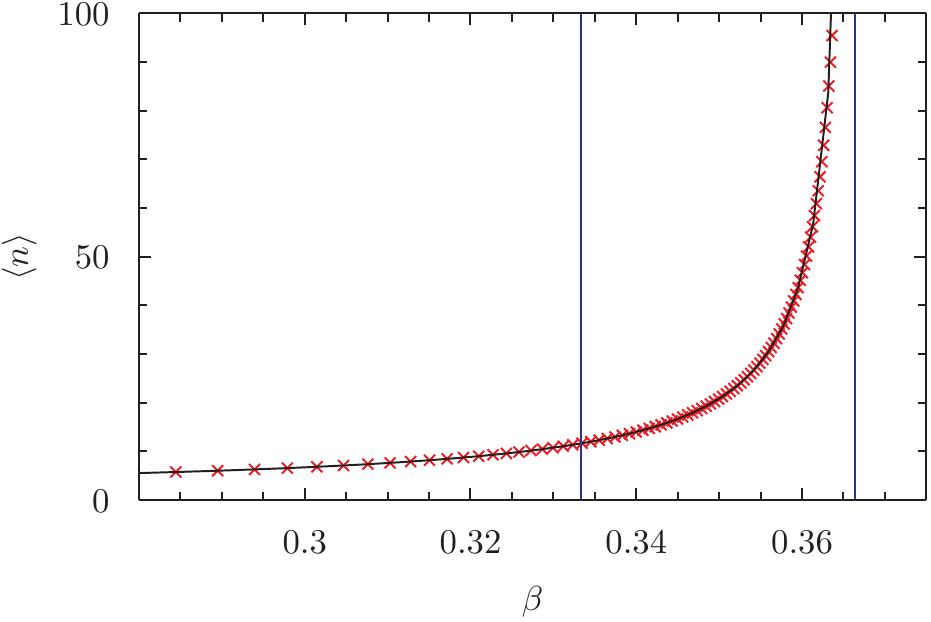}}\\
  \subfloat[$\langle a,b,c|a^2,b^2,c^2 \rangle$ sampled with $\alpha =1$.]{
    \includegraphics[width=0.5\textwidth]{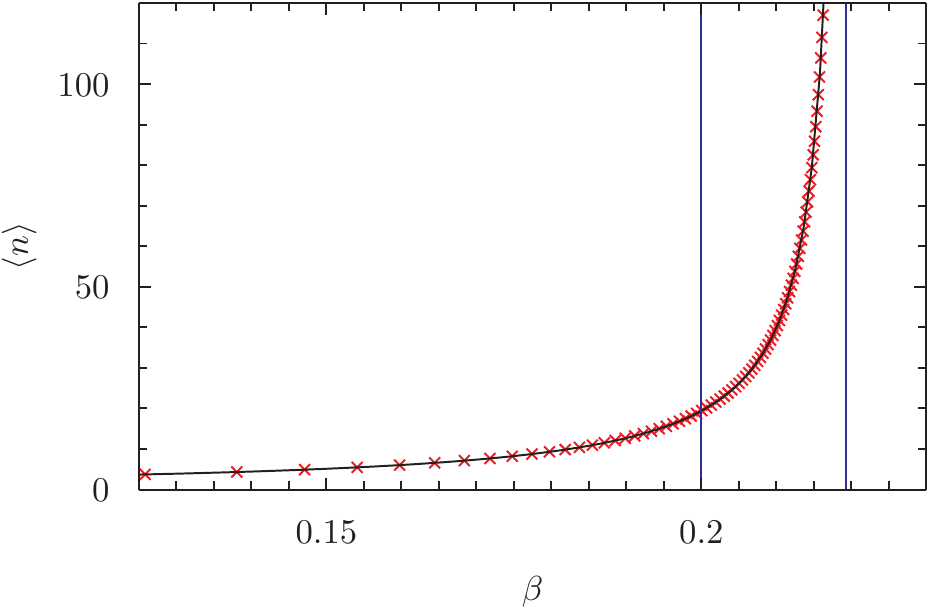}}
  \caption{Mean length of freely reduced trivial words of the indicated
groups. The sampled points are indicated with impulses, while the exact
values are given by the black line. We have used vertical lines to indicate
$\beta=\nicefrac{1}{3}, \nicefrac{1}{3}, \nicefrac{1}{5}$ (respectively)  and
also the reciprocal of the cogrowth where the statistic will diverge --- being
$0.3418821478, 0.3664068598$ and $0.2192752634$ respectively. There is excellent
agreement between the numerical and exact results, except possibly at the very
highest $\beta$ values.}
  \label{fig:kuksov plots}
\end{figure}

\begin{figure}[h]
  \centering
  \subfloat[$\langle a,b|a^2,b^3 \rangle$ sampled with $\alpha = 0,1,2,3$.]{
    \includegraphics[width=0.5\textwidth]{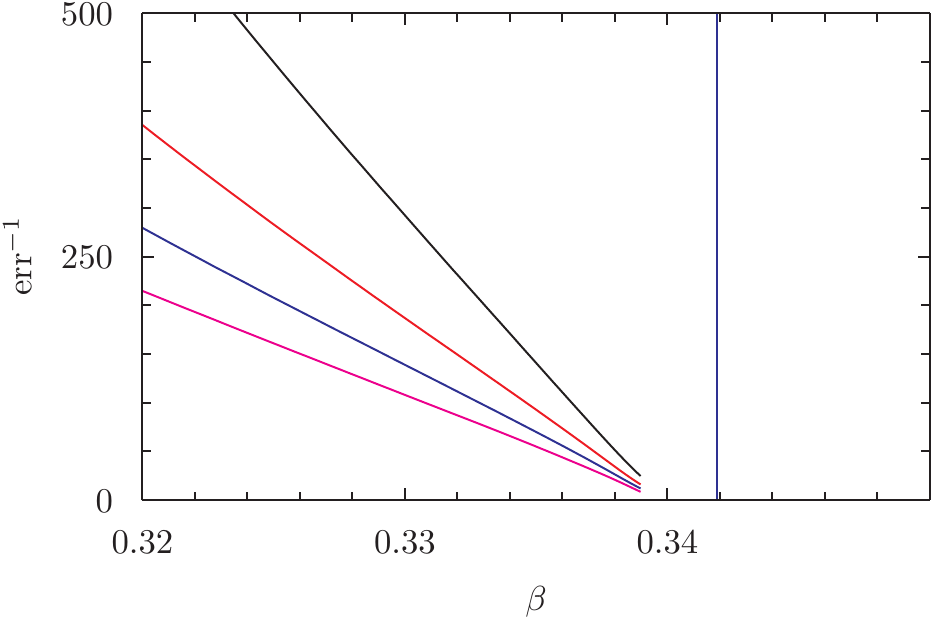}}
  \subfloat[$\langle a,b|a^3,b^3 \rangle$ sampled with $\alpha = 0,1,2,3$.]{
    \includegraphics[width=0.5\textwidth]{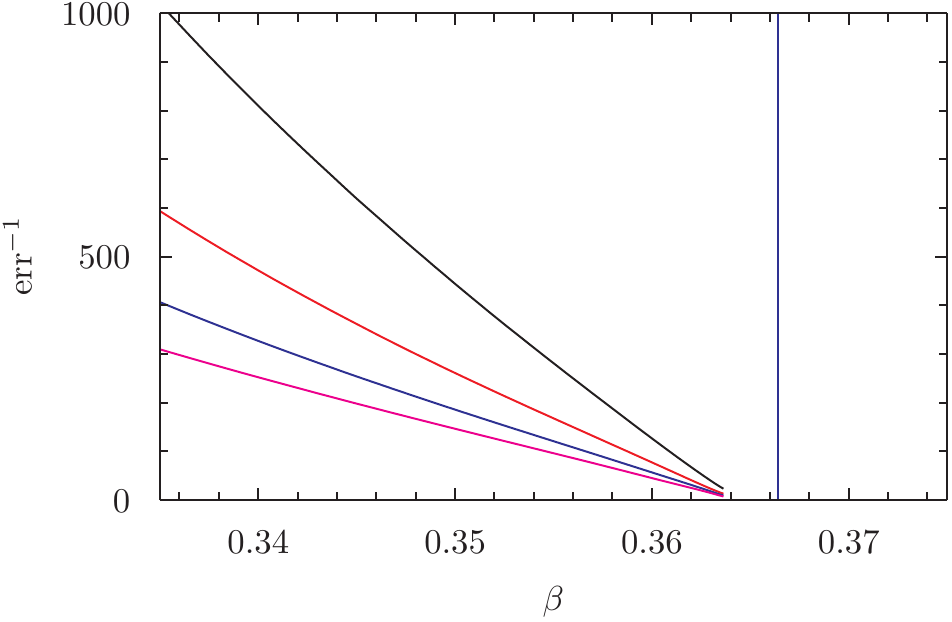}}\\
  \subfloat[$\langle a,b,c|a^2,b^2,c^2 \rangle$ sampled with $\alpha =0,1,2,3$.]{
    \includegraphics[width=0.5\textwidth]{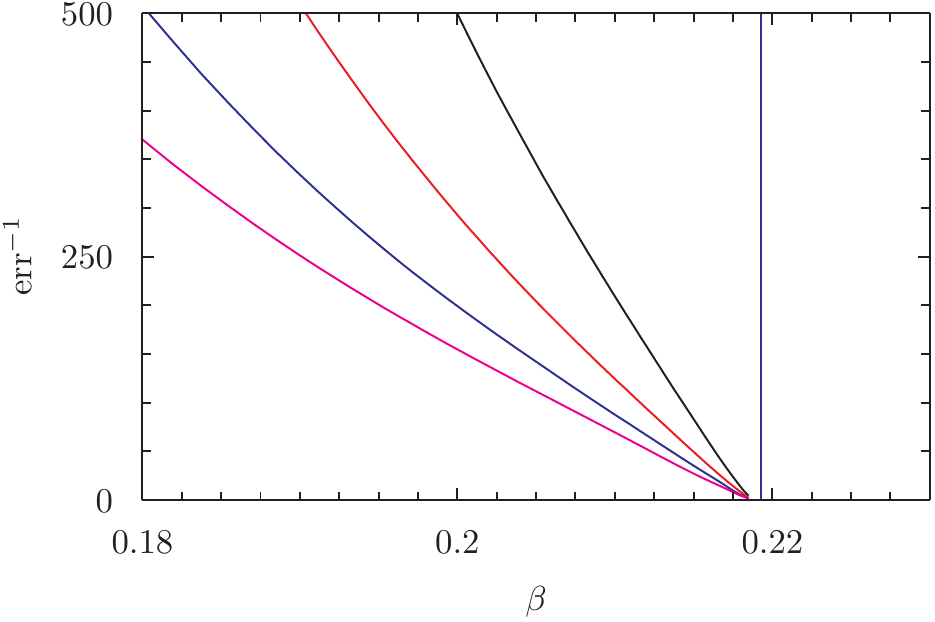}}
  \caption{The reciprocal of the estimated standard error vs $\beta$ for the
indicated groups with $\alpha=0,1,2,3$ (clockwise from top in each case).  We
see that the extrapolations of the curves intersect the $x$-axis very close, but
slightly short, of the indicated critical values of $\beta$ --- $0.3418821478,
0.3664068598$ and $0.2192752634$ respectively. Hence these give good, but
slightly low, estimates of the location of the singularities $\beta =
0.340\pm0.002, 0.365\pm0.002$ and $0.219\pm0.001$.}
  \label{fig:kuksov err plots}
\end{figure}

\begin{figure}[h]
  \centering
  \subfloat[$\mathrm{BS}(2,2)$ sampled with $\alpha = 1$.]{
    \includegraphics[width=0.5\textwidth]{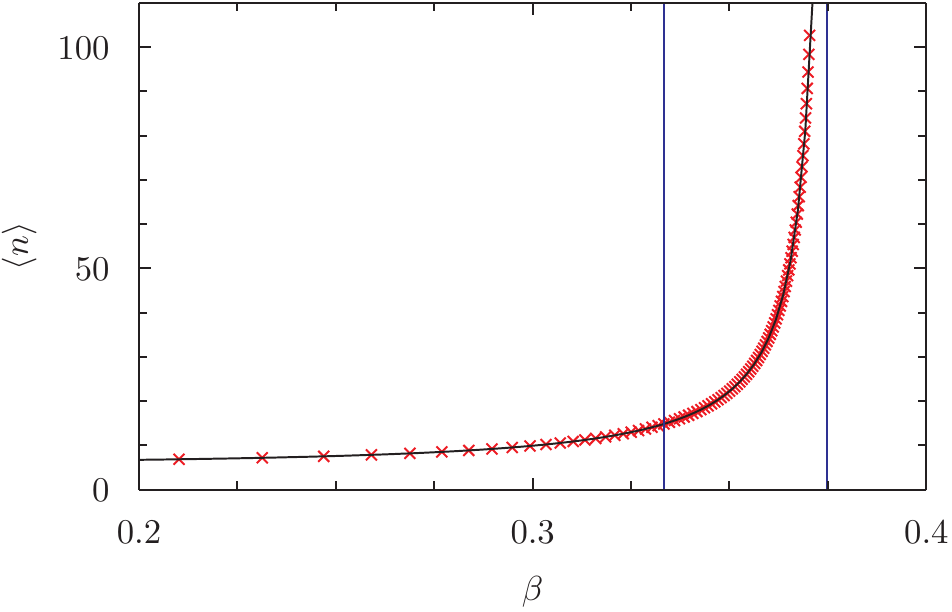}}
  \subfloat[$\mathrm{BS}(3,3)$ sampled with $\alpha = 1$.]{
    \includegraphics[width=0.5\textwidth]{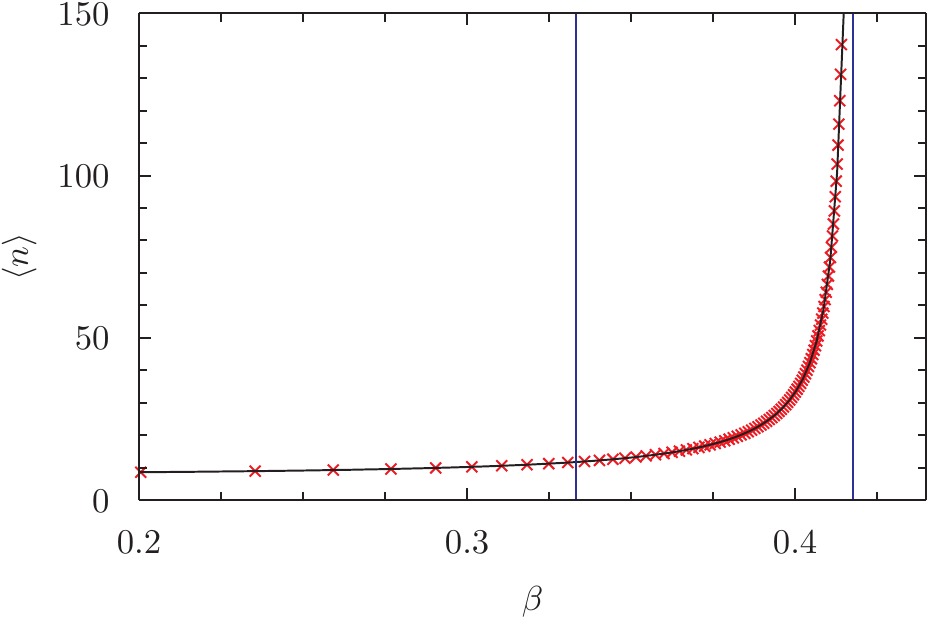}}
  \caption{Mean length of freely reduced trivial words in Baumslag-Solitar
groups $\mathrm{BS}(2,2)$ and $\mathrm{BS}(3,3)$ at different values of~$\beta$. The sampled
points are indicated with impulses, while the exact values are given by the
black line. We have used vertical lines to indicate $\beta=\nicefrac{1}{3}$ and
also the reciprocal of the cogrowth where the statistic will diverge --- being
$0.3747331572$ and $0.417525628$ respectively. We see excellent agreement
between our numerical data and the exact results, and our error bars are
smaller than the impulses.}
  \label{fig:bs2233 plots}
\end{figure}

\begin{figure}[h]
  \centering
  \subfloat[$\mathrm{BS}(2,2)$ sampled with $\alpha = 0,1,2,3$.]{
    \includegraphics[width=0.5\textwidth]{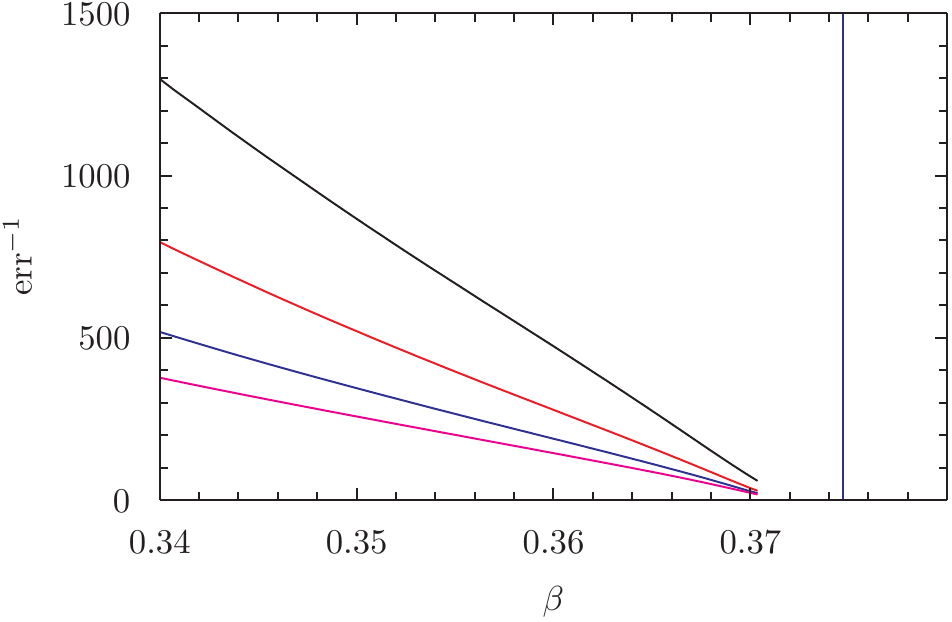}}
  \subfloat[$\mathrm{BS}(3,3)$ sampled with $\alpha = 0,1,2,3$.]{
    \includegraphics[width=0.5\textwidth]{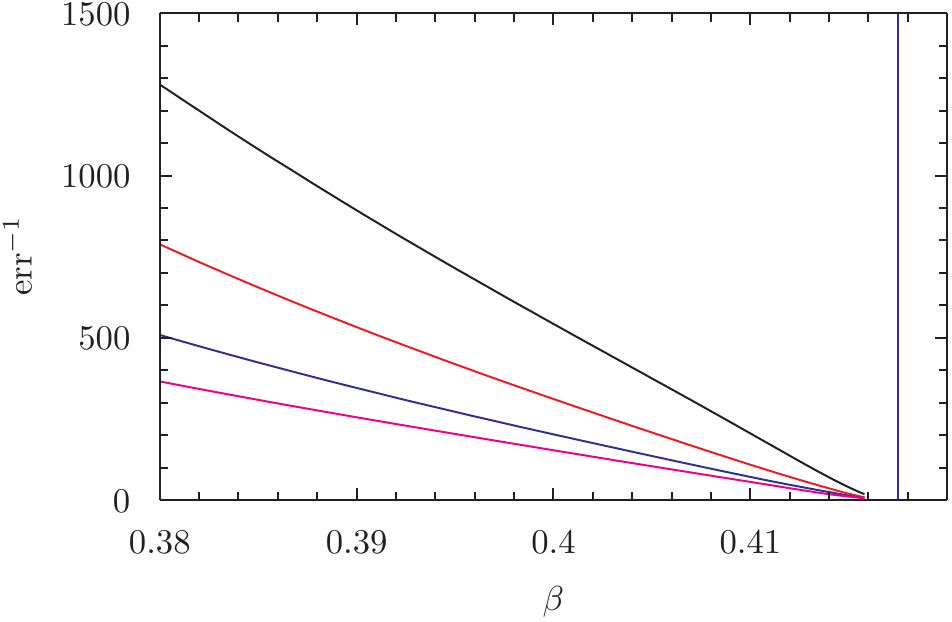}}
  \caption{The reciprocal of the estimated standard error of the mean length as
a function of $\beta$ for $\mathrm{BS}(2,2)$ and $\mathrm{BS}(3,3)$. In both plots we show 4
curves corresponding to simulations at $\alpha=0,1,2,3$ (anti-clockwise from
top) and denote the singular values --- $0.3747331572$ and $0.417525628$
respectively --- with vertical lines. Extrapolating the curves give estimates of
$\beta_c = 0.372\pm0.002$  and $0.416 \pm 0.001$ respectively.}
  \label{fig:bs2233 err plots}
\end{figure}

\begin{figure}[h]
  \centering
  \subfloat[$\mathrm{BS}(2,3)$ sampled with $\alpha=1$.] {
  \includegraphics[width=0.5\textwidth]{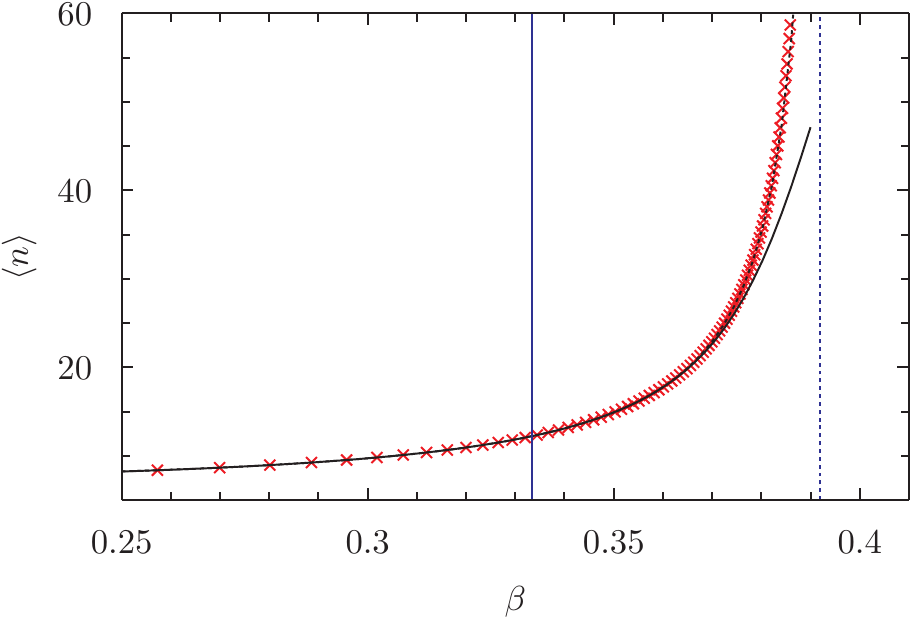} }
  \subfloat[$\mathrm{BS}(3,5)$ sampled with $\alpha=0$.] {
  \includegraphics[width=0.5\textwidth]{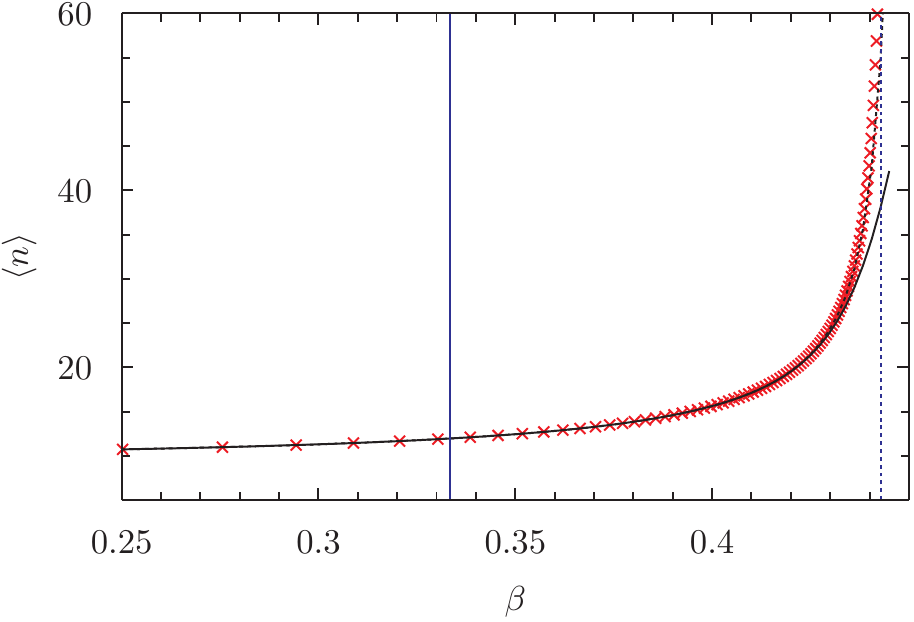} }
  \caption{The mean length of trivial words in $\mathrm{BS}(2,3)$ and $\mathrm{BS}(3,5)$ at
different values of $\beta$. We see very good for low and moderate values
of $\beta$ and by systematic errors for larger $\beta$.  Again, this error
arises from from the modest length of the exact series. The dotted curves in
these two cases indicate mean length data generated from longer but approximate
series, while the dotted vertical lines indicate the estimated critical value
of $\beta$ from series.}
  \label{fig:bs23 plots}
\end{figure}

\begin{figure}[h]
  \centering
  \subfloat[$\mathrm{BS}(2,3)$ sampled with $\alpha = 0,1,2,3$.]{
    \includegraphics[width=0.5\textwidth]{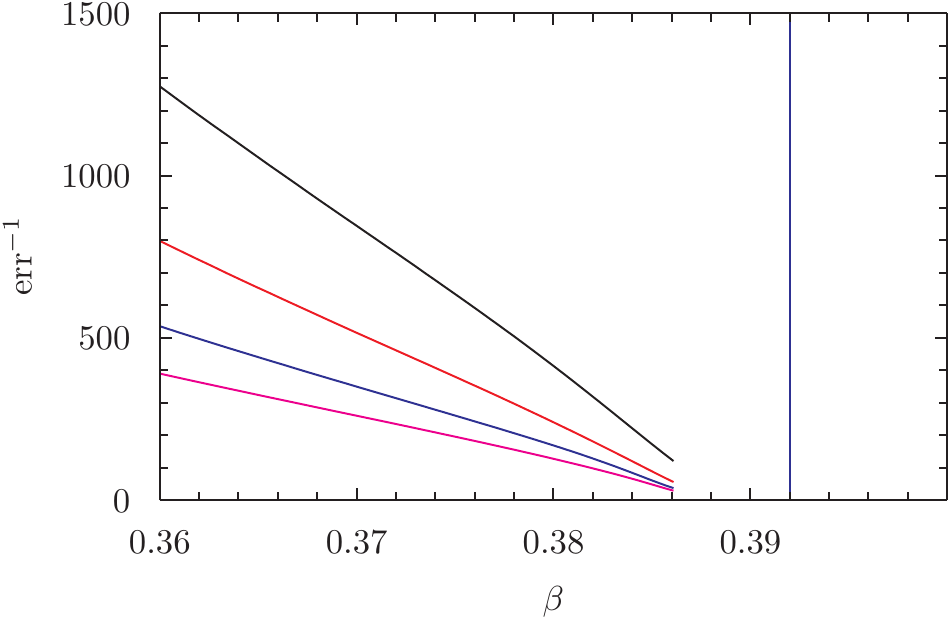}}
  \subfloat[$\mathrm{BS}(3,5)$ sampled with $\alpha = 0,1,2,3$.]{
    \includegraphics[width=0.5\textwidth]{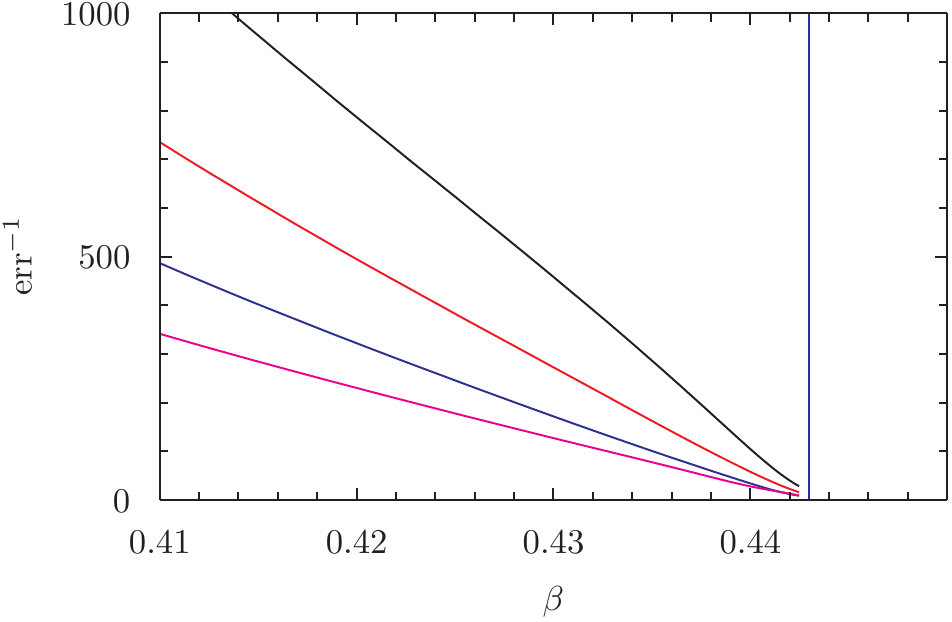}}
  \caption{The reciprocal of the estimated standard error of the mean length as
a function of $\beta$ for $\mathrm{BS}(2,3)$ and $\mathrm{BS}(3,5)$. In both plots we show 4
curves corresponding to simulations at $\alpha=0,1,2,3$ (anti-clockwise from
top). Extrapolating these curves we estimate $\beta_c = 0.388\pm0.02$ and
$0.444\pm0.002$ for $\mathrm{BS}(2,3)$ and $\mathrm{BS}(3,5)$ respectively. These are quite close
to the estimates from series of $0.393$ and $0.443$ (indicated with vertical
lines).}
  \label{fig:bs2335 err plots}
\end{figure}

\begin{figure}[h]
  \centering
  \subfloat[Standard presentation (\ref{eqnF1}) for $F$  sampled with $\alpha = 2$]{
    \includegraphics[width=0.5\textwidth]{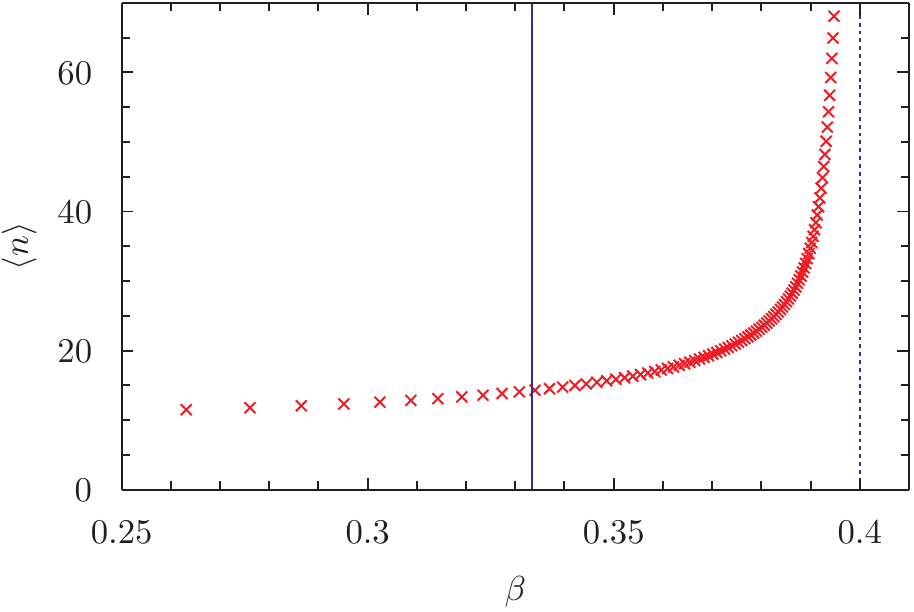}}
  \subfloat[Presentation (\ref{eqnF2})   for $F$  sampled with $\alpha = 2$]{
    \includegraphics[width=0.5\textwidth]{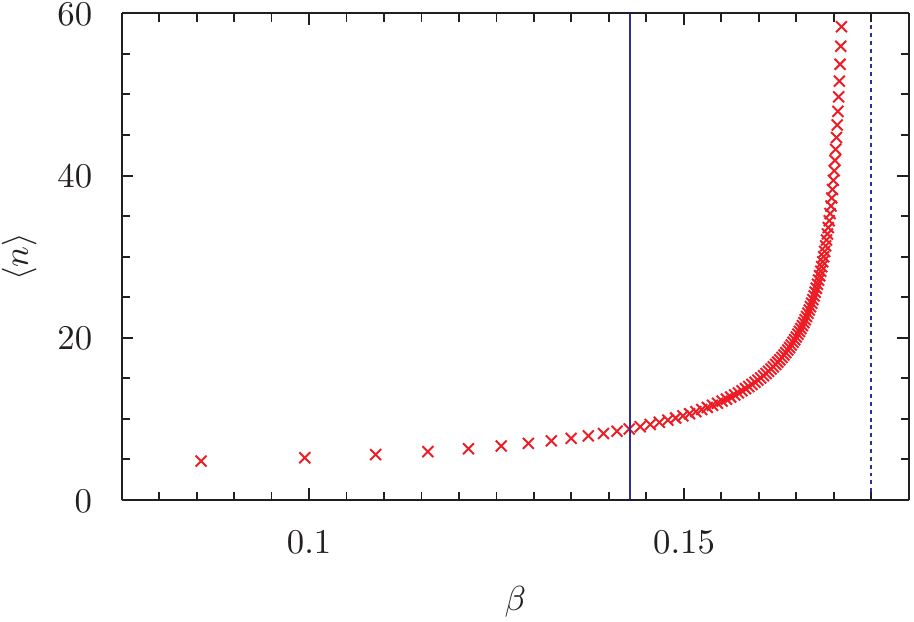}}\\
  \subfloat[Presentation (\ref{eqnF3})   for $F$ sampled with $\alpha = 1$]{
    \includegraphics[width=0.5\textwidth]{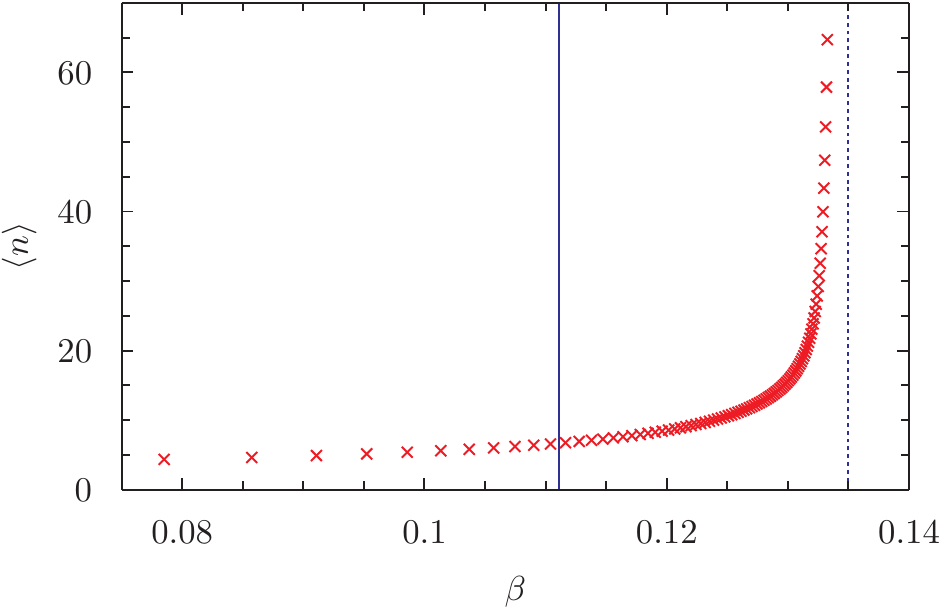}}
  \caption{Mean length of freely reduced trivial words in Thompson's group $F$
at different values of~$\beta$. The solid blue lines indicate the reciprocal of
the cogrowth of amenable groups with $k$ generators~$\beta_c =
\nicefrac{1}{2k-1}$. The dashed blue lines indicate the approximate
location of the vertical asymptote. In each case, we see that the mean length
of trivial words is finite for $\beta$-values at and slightly above~$\beta_c$.}
  \label{fig:thomp plots}
\end{figure}

\begin{figure}[h]
  \centering
  \subfloat[Standard presentation  (\ref{eqnF1})  for $F$  sampled with $\alpha = 0,1,2,3$.]{
    \includegraphics[width=0.5\textwidth]{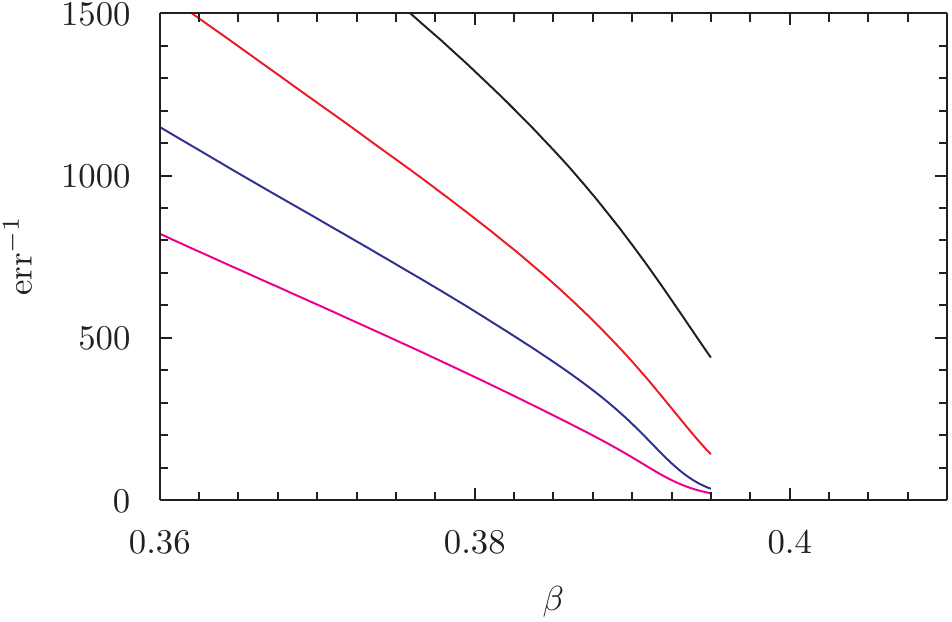}}
  \subfloat[Presentation  (\ref{eqnF2})    for $F$ sampled with $\alpha = 0,1,2,3$.]{
    \includegraphics[width=0.5\textwidth]{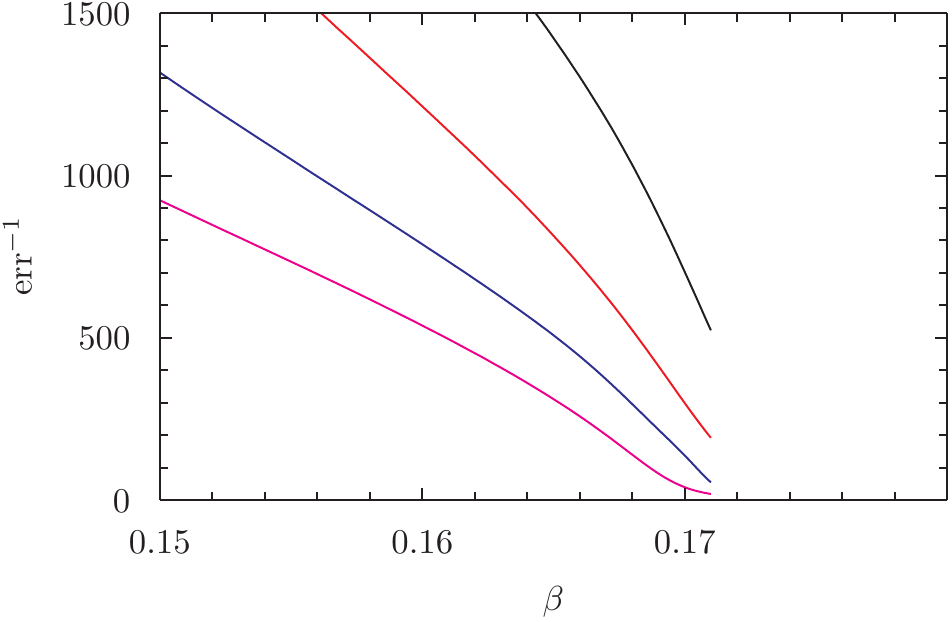}}\\
  \subfloat[Presentation  (\ref{eqnF3})   for $F$ sampled with $\alpha = 0,1,2,3$.]{
    \includegraphics[width=0.5\textwidth]{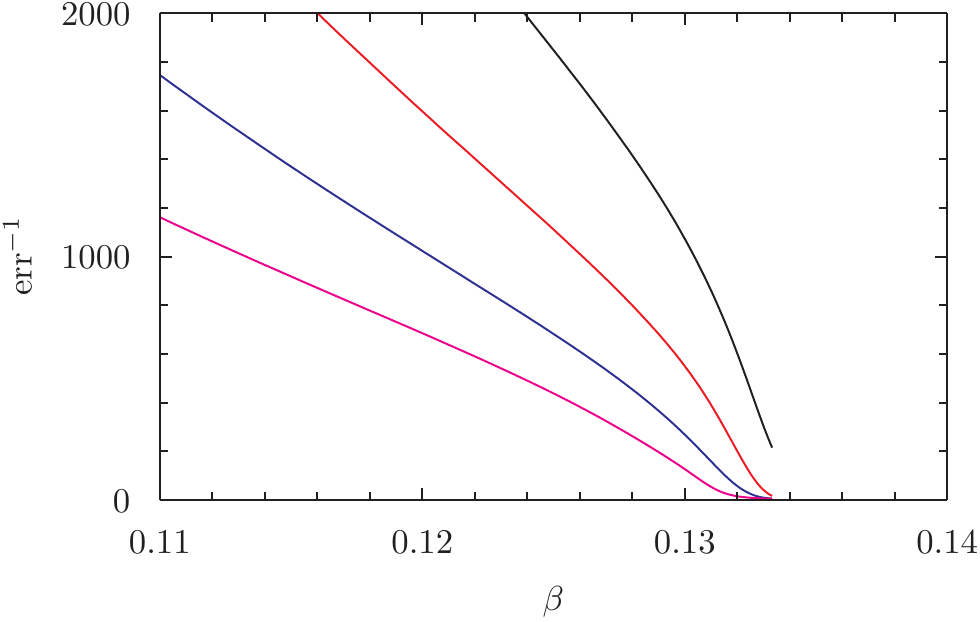}}
  \caption{The reciprocal of the estimated standard error of the mean length as
a function of $\beta$ for the three presentations of Thompson's group. In each
plots we show 4 curves corresponding to simulations at $\alpha=0,1,2,3$
(anti-clockwise from top). Extrapolating these curves leads to estimates of
$\beta_c$ of $0.395\pm0.005$, $0.172\pm 0.002$, $0.134\pm0.004$. These are all
well above the values of amenable groups.}
  \label{fig:thomp err plots}
\end{figure}

\bibliographystyle{plain}
\bibliography{refs-cogrowth}

\end{document}